\documentclass[11pt, oneside]{article}

\usepackage[T1]{fontenc}
\usepackage[a4paper,includeheadfoot,margin=2.54cm]{geometry}
\usepackage{amssymb, amsmath, mathtools, amsthm}
	\newtheorem{proposition}{Proposition}
\usepackage{graphicx}
\usepackage{ pgfplots}
	\pgfplotsset{compat=newest}
	\usepgfplotslibrary{external}
	\newlength\figureheight
	\newlength\figurewidth
\usepackage{pgfplotstable}
\usepackage{tikz}
	\usetikzlibrary{external}
	\usetikzlibrary{backgrounds}
  	\usetikzlibrary{shapes.geometric}
  	\usetikzlibrary{positioning}
	\tikzexternaldisable
\usepackage{mathrsfs}
\usepackage{hyperref}
\hypersetup{
	colorlinks,
	linkcolor={red!50!black},
 	citecolor={blue!50!black},
	urlcolor={blue!80!black}
}
\usepackage{caption}
\usepackage[caption=false]{subfig}
\usepackage{cite}
\usepackage{pifont}
\usepackage{multirow}
\usepackage{booktabs}
\usepackage{forloop}
\usepackage{pgffor}
\usepackage{enumitem}
\usepackage{algorithm}
\usepackage{algpseudocode}

\newcommand{\figref}[1]{Figure~\ref{#1}}
\newcommand{\tabref}[1]{Table~\ref{#1}}

\newcommand{\secref}[1]{Section~\ref{#1}}

\newcommand{\R}{\mathbb{R}}
\newcommand{\N}{\mathbb{N}}

\newcommand{\bx}{\boldsymbol{x}}
\newcommand{\by}{\boldsymbol{y}}

\newcommand{\TheTitle}[1]{\def\TheTitle{#1}}
\newcommand{\TheAuthors}[1]{\def\TheAuthors{#1}}

\definecolor{green}{RGB}{0,151,76}
\definecolor{linecolor4}{RGB}{233,98,30}
\definecolor{linecolor5}{RGB}{165,165,167}
\definecolor{kul}{RGB}{32,140,173}
\definecolor{BarColorOne}{RGB}{32,140,173}
\definecolor{BarColorTwo}{RGB}{255,0,0}
\definecolor{darkgrey}{rgb}{0.8,0.8,0.8}
\definecolor{lightgrey}{rgb}{0.7,0.7,0.7}
\definecolor{darkred}{RGB}{139,0,0}
\definecolor{darkgreen}{RGB}{0,100,0}
\definecolor{darkmagenta}{RGB}{180,0,180}
\definecolor{darkblue}{RGB}{0,0,190}
\definecolor{ratecolor1}{RGB}{0,100,0}
\definecolor{ratecolor2}{RGB}{139,0,0}
\definecolor{ratecolor3}{RGB}{0,0,190}

\newcolumntype{P}[1]{>{\centering\let\newline\\\arraybackslash\hspace{0pt}}m{#1}}
\newcolumntype{L}[1]{>{\raggedright\let\newline\\\arraybackslash\hspace{0pt}}m{#1}}
\newcolumntype{R}[1]{>{\raggedleft\let\newline\\\arraybackslash\hspace{0pt}}m{#1}}

\usepackage[normalem]{ulem}

\pgfdeclareplotmark{(ml)}
  {\draw[ratecolor1,fill] (0pt, 0pt) circle [radius=1.4pt];
  \draw[white,fill] (0pt, 0pt) circle [radius=0.7pt];}
\pgfdeclareplotmark{(td)}
  {
  \node [draw,scale=0.4,diamond,fill=white] at (-0.07,0) {};}
\pgfdeclareplotmark{(ft)}
  {\draw[ratecolor3,fill] (0pt, 0pt) diamond [radius=1.4pt];
  \draw[white,fill] (0pt, 0pt) diamond [radius=0.7pt];}

\pgfplotsset{
    LineStyle1/.style={color=kul,solid,line width=0.75pt},
    LineStyle2/.style={color=red,solid,line width=0.75pt},
    LineStyle3/.style={color=green,solid,line width=0.75pt},
    LineStyle4/.style={color=linecolor4,solid,line width=0.75pt},
    LineStyle5/.style={color=linecolor5,solid,line width=0.75pt},
    DottedLineStyle1/.style={LineStyle1,mark size=1,mark=*,line cap=round},
    DottedLineStyle2/.style={LineStyle2,mark size=1,mark=*,line cap=round},
    DottedLineStyle3/.style={LineStyle3,mark size=1,mark=*,line cap=round},
    DottedLineStyle4/.style={LineStyle4,mark size=1,mark=*,line cap=round},
    DottedLineStyle5/.style={LineStyle5,mark size=1,mark=*,line cap=round},
    DashedDottedLineStyle1/.style={LineStyle1,dashed,mark size=1,mark=*,line cap=round,mark options={solid}},
    DashedDottedLineStyle2/.style={LineStyle2,dashed,mark size=1,mark=*,line cap=round,mark options={solid}},
}

\newcommand{\LogLogSlopeTriangle}[5]
{

    \pgfplotsextra
    {
        \pgfkeysgetvalue{/pgfplots/xmin}{\xmin}
        \pgfkeysgetvalue{/pgfplots/xmax}{\xmax}
        \pgfkeysgetvalue{/pgfplots/ymin}{\ymin}
        \pgfkeysgetvalue{/pgfplots/ymax}{\ymax}

        \pgfmathsetmacro{\xArel}{#1}
        \pgfmathsetmacro{\yArel}{#3}
        \pgfmathsetmacro{\xBrel}{#1-#2}
        \pgfmathsetmacro{\yBrel}{\yArel}
        \pgfmathsetmacro{\xCrel}{\xBrel}

        \pgfmathsetmacro{\lnxB}{\xmin*(1-(#1-#2))+\xmax*(#1-#2)} 
        \pgfmathsetmacro{\lnxA}{\xmin*(1-#1)+\xmax*#1} 
        \pgfmathsetmacro{\lnyA}{\ymin*(1-#3)+\ymax*#3} 
        \pgfmathsetmacro{\lnyC}{\lnyA+1.1*#4*(\lnxA-\lnxB)}
        \pgfmathsetmacro{\yCrel}{\lnyC-\ymin)/(\ymax-\ymin)}
        
        \coordinate (A) at (rel axis cs:\xArel,\yArel);
        \coordinate (B) at (rel axis cs:\xBrel,\yBrel);
        \coordinate (C) at (rel axis cs:\xCrel,\yCrel);

        \draw[#5]   (A)--node[pos=0.9,yshift=1ex,xshift=0.5ex] {\scriptsize #4}
                    (B)--
                    (C)-- 
                    cycle;
    }
}

\makeatletter

\def\addlegendimage{\pgfplots@addlegendimage}
\makeatother

\newcommand{\nbshifts}{R}
\newcommand{\ridx}{r}
\newcommand{\shift}{\Delta}
\newcommand{\stochd}{s}

\newcommand{\levelcntr}{\ell}

\newcommand{\StochCntr}{j}

\newcommand{\klcntr}{\StochCntr}

\newcommand{\ExpName}{\text{E}}
\newcommand{\ExpOf}[1]{\ExpName[{#1}]}

\newcommand{\VarName}{\text{V}}
\newcommand{\VarOf}[1]{\VarName[{#1}]}

\newcommand{\SecondQoiName}{\Upsilon}

\newcommand{\MeanEstimatorWithSampleReuse}{\bar{\SecondQoiName}}

\newcommand{\MeanQMCEstimatorWithoutSampleReuse}{\MeanEstimatorWithSampleReuse_{\levelcntr}}

\newcommand{\EigenvalueName}{\theta}

\newcommand{\Eigenvalue}{\EigenvalueName_\klcntr}

\title{Recycling Samples in the Multigrid \\Multilevel (Quasi-)Monte Carlo Method}
\author{Pieterjan Robbe\thanks{KU Leuven, Department of Computer Science, NUMA Section, Celestijnenlaan 200A box 2402, 3001 Leuven, Belgium (\href{mailto:pieterjan.robbe@kuleuven.be}{pieterjan.robbe@kuleuven.be}, \href{mailto:dirk.nuyens@kuleuven.be}{dirk.nuyens@kuleuven.be}, \href{mailto:stefan.vandewalle@kuleuven.be}{stefan.vandewalle@kuleuven.be}).}
\and Dirk Nuyens\footnotemark[1]
\and Stefan Vandewalle\footnotemark[1]}
\date{}

\begin{document}

\maketitle

\begin{abstract}
The Multilevel Monte Carlo method is an efficient variance reduction technique. It uses a sequence of coarse approximations to reduce the computational cost in uncertainty quantification applications. The method is nowadays often considered to be the method of choice for solving PDEs with random coefficients when many uncertainties are involved. When using Full Multigrid to solve the deterministic problem, coarse solutions obtained by the solver can be recycled as samples in the Multilevel Monte Carlo method, as was pointed out by Kumar, Oosterlee and Dwight [\emph{Int. J. Uncertain. Quantif.}, 7 (2017), pp. 57--81]. In this article, an alternative approach is considered, using Quasi-Monte Carlo points, to speed up convergence. Additionally, our method comes with an improved variance estimate which is also valid in case of the Monte Carlo based approach. The new method is illustrated on the example of an elliptic PDE with lognormal diffusion coefficient. Numerical results for a variety of random fields with different smoothness parameters in the Mat\'ern covariance function show that sample recycling is more efficient when the input random field is nonsmooth.
\end{abstract}

\pgfplotstableread[header=false]{data/params.dat}\params
\pgfplotstableread[header=false]{plot_data/rates_timings.txt}\rates
\pgfplotstableread[header=false]{plot_data/speedup.txt}\speedup

\section{Introduction and background}\label{sec:intro}

Mathematical models in science or engineering that are subject to uncertainty often involve spatially varying input data represented as a random field. The uncertainty characterized by the input random field leads to a corresponding uncertainty in the output of the model and in any quantity of interest derived thereof. The computational goal is then to find statistics of this quantity of interest, such as an expected value or a higher-order moment.

Throughout this article, we consider the model parametric elliptic partial differential equation (PDE)
\begin{equation}\label{eq:PDE}
-\nabla\cdot a(\bx,\boldsymbol{y})\nabla u(\bx,\boldsymbol{y})=f(\bx), \quad \bx \in D
\end{equation}
where $D$ is a bounded domain in $\R^d$, $d=1,2,3$, with boundary $\partial D$. We assume deterministic boundary conditions are given. The diffusion coefficient $a(\bx,\by)$ is a random field (i.e., a random variable for which every realization is a function defined on $D$), and $\by=(y_j)_{j\geq1}$ is an infinite sequence which characterizes an instance of the random field.
In practice, this infinite number of random variables needs to be truncated.

Depending on the characteristics of the random field, and the way it is represented, many uncertain parameters can be needed to accurately model the random field. Sampling-based methods, such as the Monte Carlo method, are the preferred tool to deal with these many uncertainties. They require the solution of problem~\eqref{eq:PDE} for every draw of a sufficiently large sample set ${\boldsymbol{y}}^{(1)}, {\boldsymbol{y}}^{(2)}, \ldots, {\boldsymbol{y}}^{(N)}$, where ${\boldsymbol{y}}^{(i)}$ is a finite-dimensional vector, and $N$ is the number of samples. From this sample set, conclusions are made about the statistics of a quantity of interest $F(u)$, such as the expected value, variance or higher-order moments. Here, $F$ is a functional that is applied to the solution $u$, for example a point evaluation of $u$ or the average value over a subset of the domain $D$. Note that uncertainty in the diffusion coefficient $a(\bx,\boldsymbol{y})$ leads to uncertainty in the solution $u(\bx,\boldsymbol{y})$, and hence in the quantity of interest $F(u)$. In the remainder of the text, we make this relation explicit by using the shorthand notation $F(\boldsymbol{y})\coloneqq F(u(\bx,\boldsymbol{y}))$.

In 2008, the Multilevel Monte Carlo (MLMC) method was introduced as an alternative for standard Monte Carlo sampling, see~\cite{giles2008multilevel}. Instead of sampling the quantity of interest directly on the desired model resolution, sampling is done on the differences of a coarse-to-fine hierarchy of model approximations. Most samples will be taken on the coarse approximations, where samples are cheap, and only few samples are required on the finest approximations. This explains the huge benefit that is commonly attributed to MLMC methods.  For details, see, e.g.,~\cite{giles2015multilevel,cliffe2011multilevel,barth2011multi} and references therein.

Since then, many extensions of the multilevel method have been presented by various authors. A notable example is the Multilevel Quasi-Monte Carlo (MLQMC) method, see~\cite{giles2009multilevel,robbe2016thesis,kuo2015multilevel}, where the random sampling from the Monte Carlo method is replaced by deterministic sampling to speed up convergence. Another extension is the Multi-Index Monte Carlo (MIMC) method, see~\cite{haji2016multi}, where the multilevel hierarchy is replaced by a multidimensional hierarchy of model approximations. See also \cite{robbe2017multi} for its extension to QMC.

In~\cite{kumar2017multigrid}, a combination of Multigrid and Multilevel Monte Carlo was presented. In this method, a Full Multigrid (FMG) approach is used to solve the deterministic PDE in every sample from~\eqref{eq:PDE}. For every solution at a fine approximation, the method provides samples on all preceding coarser approximations, which can be reused in the estimator. In this work, we present an alternative approach, based on the aforementioned MLQMC method, which also allows the recycling of coarse samples. Additionally, our method comes with a mathematically sound estimate for the variance of the obtained estimator, which is also valid in the Monte Carlo based approach.

This article is organized as follows. In \secref{sec:intro_stoch}, we discuss the multilevel sampling framework in more detail and introduce the Multigrid Multilevel Monte Carlo (MG-MLMC) method. We review the practical implementation, at which point we introduce the quasi-Monte Carlo variant (MG-MLQMC), and analyze the cost of the estimator in \secref{sec:MGMLMC}. After that, we turn our focus to the PDE example outlined above. We discuss FMG in \secref{sec:intro_det}, and random field generation using the Karhunen--Lo\`eve (KL) expansion in \secref{sec:KL}. The article is concluded with various numerical experiments in~\secref{sec:numerical_results}, which confirm the theory.

\section{Multigrid Multilevel Monte Carlo}\label{sec:intro_stoch}

The classic sampling-based method is the Monte Carlo method. A Monte Carlo estimator for the expected value of the quantity of interest $\text{E}[{F}]$ is simply an average over $N$ realizations of the uncertain parameters, i.e.,
\begin{equation}
{\mathcal{Q}}^\text{MC}
\coloneqq
\frac{1}{N}
\sum_{n=1}^{N}
F({\boldsymbol{y}}^{(n)})
,\nonumber
\end{equation}
where ${\boldsymbol{y}}^{(n)}$ denotes the $n$-th independent realization. The Monte Carlo method has the well-known convergence of the root mean square error (RMSE) of order $O(1/\sqrt{N})$, for functions with a finite variance. This makes the method rather slow in real engineering applications, where many expensive samples are required to reach a certain accuracy. However, the use of multilevel sampling methods to reduce the cost of the estimator makes the Monte Carlo approach feasible in engineering practice. In the remainder of this section, we briefly discuss this multilevel extension, before introducing the MG-MLMC method.

\subsection{Multilevel Monte Carlo}\label{sec:mlmc}

Monte Carlo sampling can be combined with a multilevel idea to yield an efficient sampling method. Assume there is a hierarchy of approximations, conveniently called \emph{levels}, $F_{0}$, $F_{1}$, \ldots, $F_{L}$ that converge to the true quantity of interest $F(u)$ for $L \to \infty$. 

An example of such a hierarchy is the solution of a PDE on a mesh with decreasing grid size. Instead of estimating the quantity of interest directly on the finest and most expensive grid, the multilevel method estimates the quantity of interest on the coarsest and cheapest grid, $F_{0}$, together with a series of correction terms (or \emph{differences}) $F_{\ell}-F_{\ell-1}$, $\ell=1,\ldots,L$. These differences constitute the telescoping sum
\begin{equation}\label{eq:telsum}
  \text{E}[{F_{L}}]
  =
  \sum_{\ell=0}^L \text{E}[{F_{\ell} - F_{\ell-1}}]
  ,
\end{equation}
that lies at the basis of any multilevel method, and where we set $F_{-1} = 0$. The expectation of the differences in~\eqref{eq:telsum} can be computed using the standard Monte Carlo method, resulting in the Multilevel Monte Carlo method
\begin{equation}\label{eq:MLMC}
{\mathcal{Q}}^\text{MC}_{L}
\coloneqq
\sum_{\ell=0}^{L}
\frac{1}{N_\ell}
\sum_{n=1}^{N_\ell}
\left(
F_{\ell}({\boldsymbol{y}}^{(n)}_{\ell}) - 
F_{\ell-1}({\boldsymbol{y}}^{(n)}_{\ell})
\right)
,
\nonumber
\end{equation}
where ${\boldsymbol{y}}^{(n)}_{\ell}$ denotes the $n$-th realization of the random parameters at level $\ell$. The subscript $L$ in ${\mathcal{Q}}^\text{MC}_{L}$ denotes the finest level used in the estimator. The unknown number of samples on each level, $N_\ell\in\N$, and the total number of levels $L$, remain to be determined.

Notice that the same random numbers ${\boldsymbol{y}}^{(n)}_{\ell}$ are used to compute a sample of the difference $F_{\ell}-F_{\ell-1}$. The hope is that $F_{\ell}$ and $F_{\ell-1}$ are strongly positively correlated, and the variance of the difference $V_{\ell} \coloneqq \text{V}[{F_{\ell}-F_{\ell-1}}]$ is reduced, since then
\begin{equation}
V_{\ell} = 
\text{V}[{F_{\ell}}] + 
\text{V}[{F_{\ell-1}}] - 
2 \text{cov}(F_{\ell},F_{\ell-1}) 
\ll 
\text{V}[{F_{\ell}}] + 
\text{V}[{F_{\ell-1}}].
\nonumber
\end{equation}
Furthermore, as the level $\ell$ increases, the expectation of the difference $F_{\ell}-F_{\ell-1}$ decreases, and the variance reduction will be more pronounced.
Hence, the total number of samples $N_\ell$ decreases to zero as $\ell$ increases, and fewer samples will be required on finer levels.

In a typical MLMC simulation, the error is controlled by bounding the mean square error (MSE)
\begin{equation}\label{eq:MSE}
\text{MSE}\hspace{-2pt}\left({\mathcal{Q}}^\text{MC}_{L}\right) = 
\text{E}\hspace{-2pt}\left[
\left(
{\mathcal{Q}}^\text{MC}_{L} -
\text{E}[{F}]
\right)^2
\right] = 
\text{V}[{{\mathcal{Q}}^\text{MC}_{L}}] + 
\text{Bias}\hspace{-2pt}\left({{\mathcal{Q}}^\text{MC}_{L}}\right)^2,
\end{equation}
by a fixed tolerance $\epsilon^2$. This condition is fulfilled if both terms in~\eqref{eq:MSE} are smaller than $\epsilon^2/2$.

Since the variance of the estimator is given by the sum of the variances over the levels, and by using the classic argument of Giles from ref.~\cite{giles2008multilevel}, the cost of the estimator is minimized for a fixed variance satisfying
\begin{equation}\label{eq:vardecomp}
\text{V}[{{\mathcal{Q}}^\text{MC}_{L}}] = \sum_{\ell=0}^{L} N_\ell^{-1} V_{\ell} \leq \epsilon^2/2,
\end{equation}
if the number of samples is chosen as
\begin{equation}\label{eq:nopt}
N_\ell = \left\lceil \frac{2}{\epsilon^2}  \sqrt{\frac{V_{\ell}}{C_{\ell}}} \left( \sum_{\tau=0}^L \sqrt{V_{\tau}C_{\tau}} \right)  \right\rceil,
\end{equation}
where $C_{\ell}$ represents the cost to compute a single sample on level $\ell$.

\subsection{Formulation and main idea}\label{sec:main}

Besides efficient sampling methods, such as the MLMC method outlined above, one also requires good deterministic solvers to compute the solution $u(\bx,\boldsymbol{y})$ for every sample of the random field $a(\bx,\boldsymbol{y})$ in~\eqref{eq:PDE}. One particular choice are Multigrid methods. A Full Multigrid (FMG) method can compute a solution to discretization accurcay in $\mathcal{O}(M)$ operations, where $M$ is the number of degrees of freedom. FMG uses a hierarchy of coarse approximations, much similar to the MLMC hierarchy, to accelerate convergence on the fine grid. Further details of the FMG method will be deferred to~\secref{sec:intro_det}.

The idea of the MG-MLMC method is essentially to use the same hierarchy of coarse approximations from the FMG solver as levels in the MLMC method.  Every deterministic solution on level $\ell$ using FMG will additionally give free coarse solutions on all levels $0\leq k<\ell$, which could be recycled in the MLMC estimator, however, these coarse solutions are not uncorrelated, as required by~\eqref{eq:vardecomp}.

The MG-MLMC estimator, as it was first introduced in \cite{kumar2017multigrid}, can be written as
\begin{equation}\label{eq:MGMLMC}
{\mathcal{Q}}^\text{MC}_{L,\text{reuse}}
\coloneqq
\sum_{\ell=0}^{L}
\left(
\frac{1}{\sum_{i=\ell}^{L}N_i}
\right)
\sum_{k=\ell}^{L}
\sum_{n=1}^{N_k}
\left(
F_{\ell}({\boldsymbol{y}}^{(n)}_{k}) - 
F_{\ell-1}({\boldsymbol{y}}^{(n)}_{k})
\right)
.
\end{equation}
Here, the additional sum over $k$ represents the sample recycling, and we assume $N_0 \ge N_1 \ge \cdots \ge N_L$. Fully expanded for $L=2$, the MG-MLMC estimator reads
\begin{align}
{\mathcal{Q}}^\text{MC}_{2,\text{reuse}} &=
\frac{1}{N_0+N_1+N_2} 
\left( 
\sum_{n=1}^{N_0}
F_{0}({\boldsymbol{y}}^{(n)}_{0})
 +
\sum_{n=1}^{N_1}
F_{0}({\boldsymbol{y}}^{(n)}_{1})
 +
\sum_{n=1}^{N_2}
F_{0}({\boldsymbol{y}}^{(n)}_{2})
\right)
\nonumber
\\ 
&+
\frac{1}{N_1+N_2} 
\left(
\sum_{n=1}^{N_1}
\left(
F_{1}({\boldsymbol{y}}^{(n)}_{1}) -
F_{0}({\boldsymbol{y}}^{(n)}_{1})
\right)
 +
\sum_{n=1}^{N_2}
\left(
F_{1}({\boldsymbol{y}}^{(n)}_{2}) -
F_{0}({\boldsymbol{y}}^{(n)}_{2})
\right)
\right)\nonumber
\\
&+
\frac{1}{N_2} 
\left(
\sum_{n=1}^{N_2}
\left(
F_{2}({\boldsymbol{y}}^{(n)}_{2}) -
F_{1}({\boldsymbol{y}}^{(n)}_{2})
\right)
\right).
\nonumber
\end{align}
Equation~\eqref{eq:MGMLMC} is a sum of $L+1$ terms, where each term is given by
\begin{equation}
{\Upsilon}_\ell
\coloneqq
\left(
\frac{1}{\sum_{i=\ell}^{L}N_i}
\right)
\sum_{k=\ell}^{L}
\sum_{n=1}^{N_k}
\left(
F_{\ell}({\boldsymbol{y}}^{(n)}_{k}) - 
F_{\ell-1}({\boldsymbol{y}}^{(n)}_{k})
\right)
,
\end{equation}
which are estimators for the difference $F_{\ell}-F_{\ell-1}$, using independent random numbers ${\boldsymbol{y}}^{(n)}_{k}$. However, these estimators are not mutually uncorrelated. Hence, the variance of the estimator is
\begin{align}\label{eq:vardecomp2}
\text{V}[{{\mathcal{Q}}^\text{MC}_{L,\text{reuse}}}]
&=
\sum_{\ell=0}^{L}
\text{V}[{
{\Upsilon}_\ell
}]
+ 2
\sum_{0 \leq \ell < \tau \leq L}
\text{cov}(
{\Upsilon}_\ell
,
{\Upsilon}_\tau
)
\nonumber
\\
&=
\sum_{\ell=0}^{L}
\left(
\frac{
V_{\ell}
}
{
\sum_{i=\ell}^{L}N_i
}\right)
+ 2
\sum_{0 \leq \ell < \tau \leq L}
\rho_{\ell\tau} 
\sqrt{
\left(
\frac{
V_{\ell}
}
{
\sum_{i=\ell}^{L}N_i
}\right)
\left(
\frac{
V_{\tau}
}
{
\sum_{i=\tau}^{L}N_i
}\right)
},
\end{align}
where $\text{cov}(
{\Upsilon}_\ell
,
{\Upsilon}_\tau
)$ denotes the unknown covariance and $\rho_{\ell\tau}$ denotes the unknown correlation between ${\Upsilon}_\ell$ and ${\Upsilon}_\tau$. Compare this to equation~\eqref{eq:vardecomp}, which was a key element in allowing us to compute the required number of samples at each level. For the MG-MLMC estimator however, no such simple relationship exists, because of the covariance terms. In the next section, we will present a simple yet effective fix, that allows one to compute the variance of the estimator in a mathematically rigorous way.
Note that \cite{kumar2017multigrid} just used~\eqref{eq:vardecomp} as an approximation for \eqref{eq:vardecomp2}.

\section{Practical aspects and implementation}\label{sec:MGMLMC}

In this section, we discuss a particular implementation of the MG-MLMC method, and analyze its cost. We also present a practical algorithm for MG-MLMC simulation. 

\subsection{Obtaining an estimate for the variance}\label{sec:var}

There are several ways to obtain an estimate for the variance in equation~\eqref{eq:vardecomp2}.

\begin{itemize}
\item As a first idea, we could assume that the correlation coefficients $\rho_{\ell\tau}$ in ~\eqref{eq:vardecomp2} are all equal to 1. This allows the computation of the optimal number of samples $N_\ell$ on each level as a function of $V_{\ell}$ and $C_{\ell}$, similar to~\eqref{eq:nopt}. However, we find numerically that the variance of the estimator, obtained in this way, is a huge overestimation. Hence, the required number of samples is too large, and we do not obtain an efficient estimator.
\item Alternatively, we can use techniques known from adaptive Multilevel Monte Carlo methods, see, e.g., \cite{hoel2012adaptive,elfverson2016multilevel,eigel2016adaptive}. These methods deal with sample-dependent refinements, either using an a priori error estimate or a trial run to determine the adaptive mesh hierarchy. A variance estimate can be obtained using the debiasing technique from Rhee and Glynn, see~\cite{rhee2015unbiased}. See also~\cite{detommaso2018continuous} for details on how this can be applied in the usual MLMC setting.
\item Finally, we can use the random shifting technique known from Quasi-Monte Carlo (QMC) literature. Here, the idea is to use $\nbshifts$ independent estimators and use the sample variance over these estimators to replace~\eqref{eq:vardecomp2}. Since QMC methods have the potential to also obtain a faster convergence of the RMSE compared to standard Monte Carlo, we will choose this third approach.
\end{itemize}

\subsection{Multigrid Multilevel Quasi-Monte Carlo}\label{sec:mlqmc}

In the QMC method, the random samples from the Monte Carlo method are replaced by deterministically well-chosen sample points. These points are commonly known as Quasi-Monte Carlo points, see, e.g.,~\cite{dick2013high}. Popular choices include Sobol' sequences~\cite{sobol1967distribution} and digital nets~\cite{niederreiter1992random}. In this article, we consider rank-1 lattice rules in the $\stochd$-dimensional unit cube $[0,1]^\stochd$, where the points are given by
\begin{equation}
{\boldsymbol{t}}^{(n)} \coloneqq \frac{n\boldsymbol{z}\;\text{mod}\;N}{N} = \left\{\frac{n\boldsymbol{z}}{N}\right\},
\quad n = 0, 1, \ldots, N-1,
\nonumber
\end{equation}
with $\boldsymbol{z}\in\mathbb{Z}_N^\stochd$ a generating vector and $\{\cdot\}$ denotes the mod $1$ or fractional part operation, see, e.g.,~\cite{sloan1994lattice,nuyens2006fast}. 

To calculate the MSE in equation~\eqref{eq:MSE}, we use multiple randomly shifted lattice rules. The $\ridx$-th shifted version of the $n$-th point for level $\ell$ is then given by
\begin{equation}
{\boldsymbol{t}}^{(n)}_{\ell,\ridx} \coloneqq  \left\{\frac{n\boldsymbol{z}}{N} + {\boldsymbol{\shift}}^{(\ridx)}_{\ell}\right\},  \quad
\ridx = 1,2,\ldots,\nbshifts,
\nonumber
\end{equation}
where $\nbshifts$ is the number of shifts and ${\boldsymbol{\shift}}^{(\ridx)}_{\ell}$ are i.i.d.\ samples from a uniform distribution on $[0,1]^\stochd$. 
It can be shown that for functions with sufficient smoothness and dimensions that become progressively less important, there exist lattice rules for which the RMSE decays as $O(N^{-1+\epsilon})$, for arbitrary $\epsilon>0$, see~\cite{kuo2011quasi}.
For the numerical experiments in this paper we will consider a lognormal random field with Mat\'ern covariance which will be represented by an infinite KL-expansion. The dimensionality $\stochd$ of the QMC points is then equal to the number of terms in the truncated KL-expansion, this will be discussed in \secref{sec:KL}.
In general, QMC error bounds are specified for integration over the $\stochd$-dimensional unit cube, however the KL-expansion here uses i.i.d.\ standard normal variates.
Luckily, for randomly shifted lattice rules mapped through the inverse cumulative density of the normal distribution, as we will use here, there is an equivalent theory which shows $O(N^{-1+\epsilon})$ convergence for the RMSE, see, e.g., \cite{kuo2016application}.

We denote the multilevel estimators that uses the $r$-th randomized point set by
\begin{align}
{\mathcal{Q}}^\text{QMC}_{L,\ridx}
&\coloneqq
\sum_{\ell=0}^{L}
{\Upsilon}_{\ell,\ridx},
\quad
\text{where}
\nonumber
\end{align}
\begin{align}
{\Upsilon}_{\ell,\ridx}
&\coloneqq
\frac{1}{N_\ell}
\sum_{n=1}^{N_\ell}
\left(
F_{\ell}({\boldsymbol{y}}^{(n)}_{\ell,\ridx}) - 
F_{\ell-1}({\boldsymbol{y}}^{(n)}_{\ell,\ridx})
\right)
,  \quad
\ridx = 1,2,\ldots,\nbshifts,
\nonumber
\end{align}
with ${\boldsymbol{y}}^{(n)}_{\ell,\ridx}\coloneqq{\Phi}^{-1}({\boldsymbol{t}}^{(n)}_{\ell,\ridx})$ and where ${\Phi}^{-1}$ denotes the inverse cumulative density function (CDF) of the desired distribution of the parameter set $\boldsymbol{y}$. This provides the required mapping for the QMC points, defined on the unit cube, to the domain of interest.

The MLQMC estimator, as introduced in~\cite{giles2009multilevel,robbe2016thesis,kuo2015multilevel}, is just the average over these $\nbshifts$ estimators, i.e.,
\begin{equation}\label{eq:MLQMC}
{\bar{\mathcal{Q}}}^\text{QMC}_{L,\nbshifts}
\coloneqq
\frac{1}{\nbshifts}\sum_{\ridx=1}^{\nbshifts}
{\mathcal{Q}}^\text{QMC}_{L,\ridx}.
\nonumber
\end{equation}
The sample variance over the shifts is used as an estimate for the variance, i.e.,
\begin{equation}\label{eq:varovershifts}
\text{V}[{{\bar{\mathcal{Q}}}^\text{QMC}_{L,\nbshifts}}]
\approx 
\frac{1}{\nbshifts(\nbshifts-1)} 
\sum_{\ridx=1}^{\nbshifts}
\left(
{\mathcal{Q}}^\text{QMC}_{L,\ridx} - {\bar{\mathcal{Q}}}^\text{QMC}_{L,\nbshifts}
\right)^2.
\end{equation}
Note that in the limit $\nbshifts\rightarrow\infty$ this sample variance is (almost surely) equal to the true variance of the estimator.

A similar approach is now used for the MG-MLMC estimator from equation~\eqref{eq:MGMLMC}, i.e., we define
\begin{align}
\nonumber
{\mathcal{Q}}^\text{QMC}_{L,\ridx,\text{reuse}}
&\coloneqq
\sum_{\ell=0}^{L}
{\Upsilon}_{\ell,\ridx,\text{reuse}},
\\
\label{eq:Upsilon-reuse}
{\Upsilon}_{\ell,\ridx,\text{reuse}}
&\coloneqq
\left(
\frac{1}{\sum_{i=\ell}^{L}N_i}
\right)
\sum_{k=\ell}^{L}
\sum_{n=1}^{N_k}
\left(
F_{\ell}({\boldsymbol{y}}^{(n)}_{k,\ridx}) - 
F_{\ell-1}({\boldsymbol{y}}^{(n)}_{k,\ridx})
\right)
\quad \text{and}
\\
\label{eq:MGMLQMC}
{\bar{\mathcal{Q}}}^\text{QMC}_{L,\nbshifts,\text{reuse}}
&\coloneqq
\frac{1}{\nbshifts}\sum_{\ridx=1}^{\nbshifts}
{\mathcal{Q}}^\text{QMC}_{L,\ridx,\text{reuse}},
\end{align}
in effect turning the MG-MLMC estimator into a Multigrid Multilevel Quasi-Monte Carlo (MG-MLQMC) estimator. The variance of this estimator, that is, the equivalent of~\eqref{eq:vardecomp2}, is computed similar to equation~\eqref{eq:varovershifts}, i.e.,
\begin{equation}\label{eq:var_of_mgmlqmc}
\text{V}[{{\bar{\mathcal{Q}}}^\text{QMC}_{L,\nbshifts,\text{reuse}}}]
\approx 
\frac{1}{\nbshifts(\nbshifts-1)} 
\sum_{\ridx=1}^{\nbshifts}
\left(
{\mathcal{Q}}^\text{QMC}_{L,\ridx,\text{reuse}} - {\bar{\mathcal{Q}}}^\text{QMC}_{L,\nbshifts,\text{reuse}}
\right)^2.
\end{equation}
The MLQMC estimator~\eqref{eq:MLQMC} and the MG-MLQMC estimator~\eqref{eq:MGMLQMC} are essentially the same, the only difference being the recycling of samples on coarser levels.

\subsection{Cost analysis}\label{sec:cost_anal}

In this section, we will analyze the cost of the MG-MLQMC estimator (i.e., the MLQMC method with recycling) and compare it to the cost of the MLQMC method without recycling. For ease of exposition, assume that the quantity of interest $F_{\ell}$ at level $\ell$ corresponds to a discretization of the problem with step size~$h_\ell$.
As in~\cite{haji2014optimization}, we use a geometric hierarchy for the levels:
\begin{equation}\label{eq:grids}
h_{\ell} = \varrho \, h_{\ell+1},
\end{equation}
where $\varrho > 1$ is a scaling factor. In the PDE setting, a typical value is $\varrho = 2$.

As is usual in MLMC analysis, see, e.g., \cite[Theorem~1]{cliffe2011multilevel}, we make the following assumptions:
\begin{align}
\tag{A1}\label{A1}
\left|\text{E}[F_L-F]\right|
&\leq
c_\alpha \, h_L^\alpha,
\\
\tag{A2}\label{A2}
V_{\ell}
&\leq
c_\beta \, h_{\ell-1}^\beta,
\\
\tag{A3}\label{A3}
C_{\ell}
&\leq
c_\gamma \, h_{\ell}^{-\gamma},
\end{align}
for some constants $c_\alpha$, $c_\beta$ and $c_\gamma$.
Here, $\alpha > 0$ is the rate of the decrease of the bias in terms of~$L$, $\beta > 0$ is the rate of the decrease of the variance of the differences in terms of the levels~$\ell$ and $\gamma > 0$ gives the increase in cost of solving the PDE.
Additionally, we make an assumption on the convergence rate of the lattice rule used in the MLQMC estimator. That is, we define
\begin{equation}\label{eq:meanofqmcestimators}
\bar{\Upsilon}_{\ell}
\coloneqq
\frac{1}{\nbshifts} \sum_{\ridx=1}^{\nbshifts} {\Upsilon}_{\ell,\ridx},
\end{equation}
the approximation for the expected value of the difference $F_{\ell}-F_{\ell-1}$ using a randomly shifted lattice rule, averaged over all shifts, and assume
\begin{align}
\tag{A4}\label{A4}
\text{V}[{{\bar{\Upsilon}_{\ell}}}]
&\leq
c_\lambda \, { N_\ell^{-1/\lambda} \, V_{\ell}},
\end{align}
for some constant $c_\lambda > 0$ with $1/\lambda > 0$ the rate of convergence of the randomized lattice rule. Assumption~\eqref{A4} corresponds to assumption (M2) in~\cite[Theorem~1]{kuo2015multilevel}.

We proceed as in~\secref{sec:mlmc} and minimize the cost of the MLQMC estimator without recycling given the variance constraint, assuming the $N_\ell$ are continuous variables:
\begin{align}
&\underset{N_\ell}{\text{min}} \sum_{\ell=0}^{L} N_\ell \, c_\gamma \, h_\ell^{-\gamma} \nonumber\\
&\text{s.t.} \;\sum_{\ell=0}^{L} c_\lambda \, N_\ell^{-1/\lambda} \, c_\beta \, h_\ell^{\beta} \leq \epsilon^2/2.\nonumber
\end{align}
To this end, we consider the Lagrangian
\begin{equation}
\mathcal{L}(N_0,N_1,\ldots,N_L,\zeta) 
\coloneqq
c_\gamma \sum_{\ell=0}^{L} N_\ell \, h_\ell^{-\gamma} +
\zeta
\left(
c_\lambda c_\beta \sum_{\ell=0}^{L} N_\ell^{-1/\lambda} \, h_\ell^{\beta} - \epsilon^2/2
\right)
\nonumber
\end{equation}
where $\zeta$ is the Lagrange multiplier, and look for its stationary point. 
This leads to the first-order, necessary optimality conditions 
\begin{align}
\frac{\partial\mathcal{L}}{\partial N_\ell} &= c_\gamma h_\ell^{-\gamma} - c_\lambda c_\beta \frac{\zeta}{\lambda} N_\ell^{-1/\lambda-1} \, h_\ell^{\beta} = 0 \quad \text{for } \ell = 0,\ldots,L, \quad\text{and} \label{eq:firstopt} \\
\frac{\partial\mathcal{L}}{\partial \zeta} &= c_\lambda c_\beta\sum_{\ell=0}^{L} N_\ell^{-1/\lambda} \, h_\ell^{\beta} - \epsilon^2/2 = 0.\nonumber
\end{align}
Rearranging~\eqref{eq:firstopt}, we find that
\begin{equation}\label{eq:nopt2}
N_\ell = 
\left(
\frac{c_\lambda c_\beta}{c_\gamma}
\frac{\zeta}{\lambda}
h_\ell^{\beta+\gamma}
\right)^{\frac{\lambda}{\lambda+1}}.
\end{equation}

Note that this equation for $N_\ell$ is of limited practical use, because estimates for $\lambda$, and hence for the convergence rate   $1/\lambda$ of the MSE of the lattice rule, are very unreliable~\cite{kuo2015multilevel}. A simple algorithm that adaptively increases the number of samples $N_\ell$, but does not require knowledge of $\lambda$, will be presented in~\secref{sec:practical}.

\begin{proposition}\label{prop:cost_reduction}
The cost reduction factor of the Multigrid Multilevel (Quasi-)Monte Carlo algorithm using a geometric refinement $h_\ell = \varrho \, h_{\ell+1}$ is given by
\begin{align*}
\frac{
  \mathrm{cost}({\bar{\mathcal{Q}}}^\mathrm{QMC}_{L,\nbshifts,\mathrm{reuse}})
}{
  \mathrm{cost}({\bar{\mathcal{Q}}}^\mathrm{QMC}_{L,\nbshifts})
}
&=
1- \varrho^{-(\beta+\gamma) \lambda / (\lambda+1)}
,
\end{align*}
where $\beta$ is the rate of the decrease in variance and $\gamma$ the rate of the increase in cost per level $\ell$, see assumptions~\eqref{A2} and~\eqref{A3}, and $1/\lambda$ is the convergence rate of the randomized quasi-Monte Carlo method, see~\eqref{A4}, or $\lambda=1$ when using Monte Carlo.
\end{proposition}
\begin{proof}
Suppose the algorithm requires one to take $N_\ell$ samples on level~$\ell$. In the MG-MLQMC method, where samples from higher levels $\ell+1,\ldots,L$ are recycled, we already have $N_{\ell+1}$ samples available on level $\ell$. The remaining 
 \begin{equation}
 N'_\ell \coloneqq N_{\ell} - N_{\ell+1} \nonumber
 \end{equation}
samples still need to be computed. From~\eqref{eq:nopt2}, we find that
\begin{equation}
N_{\ell+1} = \left(\varrho^{-(\beta+\gamma)}\right)^{\frac{\lambda}{\lambda+1}} N_{\ell}, \nonumber
\end{equation}
and hence
\begin{equation}
N'_\ell = \left(1-\left(\varrho^{-(\beta+\gamma)}\right)^{\frac{\lambda}{\lambda+1}}\right) N_{\ell}. \nonumber
\end{equation}
Therefore,
\begin{align*}
\text{cost}({\bar{\mathcal{Q}}}^\text{QMC}_{L,\nbshifts,\text{reuse}})
&=
\sum_{\ell=0}^{L} R \,  N'_\ell \, C_{\ell}
\\
&=
\left(1-\left(\varrho^{-(\beta+\gamma)}\right)^{\frac{\lambda}{\lambda+1}}\right) \sum_{\ell=0}^{L} R \, N_{\ell} \, C_{\ell}
\\
&= 
\left(1-\left(\varrho^{-(\beta+\gamma)}\right)^{\frac{\lambda}{\lambda+1}}\right) \text{cost}({\bar{\mathcal{Q}}}^\text{QMC}_{L,\nbshifts})
.
\end{align*}
\end{proof}

The recycling of coarse samples is more efficient when the rates $\beta$ and $\gamma$ are small, corresponding to a slow decay of the number of samples $N_\ell$ as $\ell$ increases. This is consistent with the results in~\cite{kumar2017multigrid}. There is an additional effect in the cost reduction factor that comes from using QMC points. When using rank-1 lattice rules, $\lambda$ varies between $1$ (MC) and $1/2$ (optimal rank-1 lattice rule), hence $1/3 < \lambda/(\lambda+1) \leq 1/2$, where the lower bound is obtained for the lattice rule with the fastest convergence.

In conclusion, when the decay of the variance is slow, i.e., $\beta$ is small, because of a lack of smoothness in the problem, or when the QMC method has good performance, i.e., $1/\lambda$ is large, the sequence of $N_\ell$ given in~\eqref{eq:nopt2} will decrease slowly and the benefit of the sample recycling will be apparent. Conversely, if the decay of the variance is fast, or when the QMC method has bad performance, the sequence of $N_\ell$ will decrease very rapidly and there is not much to be gained by reusing samples across the different levels. This is indeed what is confirmed by our numerical results in~\secref{sec:numerical_results}.

\subsection{Practical algorithm}\label{sec:practical}

In \cite{giles2009multilevel}, a simple doubling algorithm is presented for MLQMC simulation. With a slight modification, this algorithm can also be used for MG-MLQMC simulation (see Algorithm~\ref{alg:alg}). First, let us define the analogue of~\eqref{eq:meanofqmcestimators},
\begin{equation}\label{eq:Upsilon-reuse-bar}
\bar{\Upsilon}_{\ell,\text{reuse}}
\coloneqq
\frac{1}{\nbshifts} \sum_{\ridx=1}^{\nbshifts} {\Upsilon}_{\ell,\ridx,\text{reuse}}.
\end{equation}
This is the approximation for the expected value of the difference $F_{\ell}-F_{\ell-1}$, using the MG-MLQMC estimator, averaged over all random shifts $\ridx=1,2,\ldots,\nbshifts$. The sample variance of $\bar{\Upsilon}_{\ell,\text{reuse}}$ with respect to the random shifts is a measure for the contribution of that level $\ell$ to the total variance of the estimator, i.e., define
\begin{equation}\label{eq:sample_variances}
\mathcal{V}_\ell
\coloneqq 
\frac{1}{\nbshifts(\nbshifts-1)}
\sum_{\ridx=1}^{\nbshifts}
\left(
{\Upsilon}_{\ell,\ridx,\text{reuse}} - \bar{\Upsilon}_{\ell,\text{reuse}}
\right)^2.
\end{equation}

The idea of the algorithm is to double the number of samples where the ratio of the sample variance $\mathcal{V}_\ell$ and the cost $C_{\ell}$, is largest, until the total variance of the estimator is smaller than the requested tolerance $\epsilon^2/2$. The main difference between the MLQMC algorithm and the MG-MLQMC algorithm in Algorithm~\ref{alg:alg} is on line~\ref{alg:line:FMGslove}, where the additional (free) coarse samples from the FMG solver are added.

The algorithm is level-adaptive, in the sense that extra levels are added to the estimator only if required to reach a certain accuracy. The bias of the estimator, defined in equation~\eqref{eq:MSE}, is computed as
\begin{equation}\label{eq:bias}
\text{Bias}\hspace{-2pt}\left({{\bar{\mathcal{Q}}}^\text{QMC}_{L,\nbshifts,\text{reuse}}}\right) \approx \frac{\left|\bar{\Upsilon}_{L,\text{reuse}}\right|}{\varrho^\alpha+1},
\end{equation}
using an estimate for~$\alpha$ obtained from a linear fit through the expected value of the differences, $\bar{\Upsilon}_{\ell,\text{reuse}}$, on all available levels, cf.~\eqref{A1}, see~\cite{giles2008multilevel}. The condition on $L$ on line~\ref{alg:line:condition} of Algorithm~\ref{alg:alg} ensures that enough levels are available to estimate the rate~$\alpha$.

\begin{algorithm}[t]\small
\begin{algorithmic}[1]
\Procedure{MG-MLQMC}{$\epsilon$}
\State set defaults $N_\ell =  N_\ell^\textrm{old} = 0$, $\Upsilon_{\ell,r,\textrm{reuse}} = \bar{\Upsilon}_{\ell,\text{reuse}} = 0$, $\mathcal{V}_\ell = +\infty$, $C_\ell = 0$ for all $\ell$ and $r$
\State{}
\Procedure{SampleAndSolve}{$\ell$} \Comment{FMG sample routine on level $\ell$}
\State start timer
\For{$\ridx = 1,\ldots,\nbshifts$}
\For{$n = N_\ell^\textrm{old},\ldots,N_\ell-1$}
\State construct random field based on $\boldsymbol{y}_{\ell,r}^{(n)}$
\State solve PDE using multigrid resulting in solutions for all levels $\tau=0,\ldots,\ell$ \label{alg:line:FMGslove}
\EndFor
\State update $\Upsilon_{\tau,r,\textrm{reuse}}$ using~\eqref{eq:Upsilon-reuse} for all $\tau=0,\ldots,\ell$
\EndFor
\State stop timer and update $C_\ell$ based on timing for $(N_\ell - N_\ell^\textrm{old}) R$ solves
\State update $\bar{\Upsilon}_{\tau,\text{reuse}}$ using~\eqref{eq:Upsilon-reuse-bar} for all $\tau=0,\ldots,\ell$
\State update $\mathcal{V}_\tau$ using~\eqref{eq:sample_variances} for all $\tau=0,\ldots,\ell$
\State set $N_\ell^\textrm{old} = N_\ell$
\EndProcedure
\State{}
\State set $L=-1$
\While{$L<2$ \textbf{or} bias estimate~\eqref{eq:bias} $ > \epsilon/\sqrt{2}$} \label{alg:line:condition}
\State set $L=L+1$
\While{variance estimate \eqref{eq:var_of_mgmlqmc} $> \epsilon^2/2$}
\State find $\ell \le L$ with maximum ratio $\mathcal{V}_\ell/C_{\ell}$ and set $N_\ell=\mathrm{max}(1, 2 N_\ell)$
\State \textbf{call} \textsc{SampleAndSolve}($\ell$)
\EndWhile
\EndWhile
\State{}
\EndProcedure
\end{algorithmic}
\caption{\label{alg:alg}MG-MLQMC simulation}
\end{algorithm}

\section{Full Multigrid}\label{sec:intro_det}

Before turning to some numerical experiments, we give some comments on the Full Multigrid method.  The basics of a Multigrid method for the iterative solution of discretizations of elliptic PDEs is the correction scheme. A two-grid correction scheme can be written as the following procedure:
\begin{enumerate}
\item Use an iterative method with $\mu_{1}$ iterations on ${A}_{\ell}{\boldsymbol{u}}_{\ell}={\boldsymbol{b}}_{\ell}$ to approximate ${\boldsymbol{u}}_{\ell}$.
\item Restrict the residual ${\boldsymbol{r}}_{\ell}={\boldsymbol{b}}_{\ell}-{A}_{\ell}{\boldsymbol{u}}_{\ell}$ to the coarse grid by ${\boldsymbol{r}}_{{\ell-1}}=R_{\ell}^{{\ell-1}}{\boldsymbol{r}}_{\ell}$.
\item Solve ${A}_{{\ell-1}}{\boldsymbol{e}}_{{\ell-1}}={\boldsymbol{r}}_{{\ell-1}}$.
\item Interpolate the error ${\boldsymbol{e}}_{{\ell-1}}$ back to the fine grid by ${\boldsymbol{e}}_{\ell}=I_{{\ell-1}}^{\ell}{\boldsymbol{e}}_{{\ell-1}}$.
\item Correct the fine-grid approximation by ${\boldsymbol{u}}_{\ell}={\boldsymbol{u}}_{\ell}+{\boldsymbol{e}}_{\ell}$.
\item Use an iterative method with $\mu_{2}$ iterations on ${A}_{\ell}{\boldsymbol{u}}_{\ell}={\boldsymbol{b}}_{\ell}$ to approximate ${\boldsymbol{u}}_{\ell}$.
\end{enumerate}
Here, ${A}_{\ell}$ is the original matrix that results from a discretization of a sample of~\eqref{eq:PDE} at level $\ell$. Matrices $R_{\ell}^{{\ell-1}}$ and $I_{{\ell-1}}^{\ell}$ are the intergrid operators, transferring quantities from the fine grid to the coarse grid and vice versa. The integers $\mu_{1}$ and $\mu_{2}$ are cycling parameters that control the number of iterations performed before and after the coarse grid correction. Typical values are $\mu_{1}=2$ and $\mu_{2}=1$. The iterative method used in step 1 and 6 is called a smoother, because it removes high frequency components from the error, see, e.g.,~\cite{hackbusch1994iterative}. A single application of the smoother is known as a relaxation step. 

A multilevel iterative process is eventually obtained by repeating this two-grid process of relaxation and coarsening on ever coarser grids. This is known as a Multigrid V-cycle. Details on the Multigrid method and why it works, can be found in, e.g., ~\cite{briggs2000multigrid,trottenberg2000multigrid,hackbusch2013multi}.

Multigrid methods come in many flavors. One notable example is Full Multigrid, which combines the correction scheme with nested iteration. Starting from the coarsest grid on level $\ell$, a good initial guess for the next grid on level $\ell+1$ is obtained by interpolation of the inexpensive-to-compute coarse-grid solution. After that, a number of multilevel correction cycles are performed. We denote this number by $\mu_{0}$. The process is then repeated until a good initial guess for an iterative method on the finest grid (level $L$) is obtained. As a result, we do not only solve the PDE on the finest grid, but we also get solutions on all coarser grids, as required. The sequence of levels visited, starting from the finest level, are visualized in~\figref{fig:FMG}. The discrete solutions ${\boldsymbol{u}}_{0}, {\boldsymbol{u}}_{1}, \ldots,{\boldsymbol{u}}_{\ell}$ that are eventually used in the estimator are encircled on the figure.

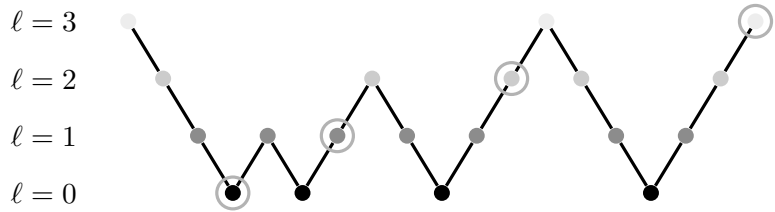
\begin{figure}
\centering
	\tikzexternalenable
	\tikzsetnextfilename{FMGgrids}%
	\begin{tikzpicture}
	\tikzset{gen/.style={node distance=6mm and 3mm,inner sep=0pt}}
	\tikzset{dot/.style={gen,circle,minimum size=4pt,text width=6pt}}
	\tikzset{black/.style={fill=black!100!white}}
	\tikzset{dark/.style={fill=black!45!white}}
	\tikzset{normal/.style={fill=black!20!white}}
	\tikzset{light/.style={fill=black!7!white}}
		\node[dot,light] (1) {};
		\node[dot,normal,below right=of 1] (2) {};
		\node[dot,dark,below right=of 2] (3) {};
		\node[dot,black,below right=of 3] (4) {};
		\node[draw=black!30!white,very thick,circle,minimum size=12pt] at (4) {};
		\node[dot,dark,above right=of 4] (5) {};
		\node[dot,black,below right=of 5] (6) {};
		\node[dot,dark,above right=of 6] (7) {};
		\node[draw=black!30!white,very thick,circle,minimum size=12pt] at (7) {};
		\node[dot,normal,above right=of 7] (8) {};
		\node[dot,dark,below right=of 8] (9) {};
		\node[dot,black,below right=of 9] (10) {};
		\node[dot,dark,above right=of 10] (11) {};
		\node[dot,normal,above right=of 11] (12) {};
		\node[draw=black!30!white,very thick,circle,minimum size=12pt] at (12) {};
		\node[dot,light,above right=of 12] (13) {};
		\node[dot,normal,below right=of 13] (14) {};
		\node[dot,dark,below right=of 14] (15) {};
		\node[dot,black,below right=of 15] (16) {};
		\node[dot,dark,above right=of 16] (17) {};
		\node[dot,normal,above right=of 17] (18) {};
		\node[dot,light,above right=of 18] (19) {};
		\node[draw=black!30!white,very thick,circle,minimum size=12pt] at (19) {};
		\node[left=of 1,anchor=center] (20) {$\levelcntr=3$};
		\node[anchor=center] at (20|-2) {$\levelcntr=2$};
		\node[anchor=center] at (20|-3) {$\levelcntr=1$};
		\node[anchor=center] at (20|-4) {$\levelcntr=0$};
		
		\begin{scope}[on background layer]
			\draw[very thick] (1) -- (2) -- (3) -- (4) -- (5) -- (6) -- (7) -- (8) -- (9) -- (10) -- (11) -- (12) -- (13) -- (14) -- (15) -- (16) -- (17) -- (18) -- (19);
		\end{scope}
\end{tikzpicture}%
	\tikzexternaldisable

\caption{\label{fig:FMG}Schematic representation of the sequence of grids visited in the FMG solver for $\mu_{0}=1$ and $\ell=3$. The coarse solutions at the end of every V-cycle, encircled on the figure, are eventually reused in the MLMC estimator.}
\end{figure}

A large number of efficient (black-box) multigrid solvers can be found in the literature, see, e.g., \cite{ruge5algebraic,yang2002boomeramg}. Typically, these codes do not provide access to the solutions on coarser grid levels. However, a multigrid solver for elliptic problems is easy to construct.

\section{Random field generation using the KL expansion}\label{sec:KL}

In this section, we focus on how to obtain a sample of the diffusion coefficient in~\eqref{eq:PDE}, represented as the random field $a(\bx,\boldsymbol{y})$, given the set of parameters~$\boldsymbol{y}$. We assume the diffusion coefficient is derived from a Gaussian random field $z(\bx,\boldsymbol{y})$, that is fully characterized by its mean $\bar{z}(\bx)$ and covariance function $C(\bx,\bx')$. In our numerical experiments later, we will use a so-called lognormal random field $a(\bx,\boldsymbol{y}) = \exp(z(\bx,\boldsymbol{y}))$, see~\cite{cliffe2011multilevel}. This ensures that equation~\eqref{eq:PDE} has physical meaning, i.e., the diffusion coefficient remains positive for all samples.

The KL expansion is a common approach to represent $z(\bx,\boldsymbol{y})$, see, e.g.,~\cite{lemaitre2010spectral}. The idea is to express the random field via a sequence of independent random variables $\boldsymbol{y}$ in the following way:
\begin{equation}\label{eq:KL}
z(\bx,\boldsymbol{y})
=
\sum_{j=1}^{\infty}
y_j
\sqrt{\theta_j} 
\psi_j
(
\bx
).
\end{equation}
The eigenvalues $\theta_j$ and eigenfunctions $\psi_j(\bx)$ are the solutions of the eigenvalue problem
\begin{equation}\label{eq:inteq}
\int_D C(\bx,\bx') \psi_j(\bx') \mathrm{d}\bx'
=
\theta_j \psi_j(\bx).
\end{equation}
Truncating the series in~\eqref{eq:KL} at the $\stochd$-th term gives the truncated KL expansion
\begin{equation}\label{eq:truncatedKL}
z_\stochd(\bx,\boldsymbol{y})
=
\sum_{j=1}^\stochd
y_j
\sqrt{\theta_j} 
\psi_j
(
\bx
).
\end{equation}

In the special case of Gaussian random fields, it can be shown that the random parameters $\boldsymbol{y}$ in~\eqref{eq:KL} are i.i.d. and standard normally distributed.

It is well-known that the KL expansion is optimal in the mean square sense, i.e., when truncated after a finite number of terms $\stochd$, the resulting approximation $z_\stochd(\bx,\boldsymbol{y})$ minimizes the MSE, see for instance~\cite{ghanem1991stochastic}. The mean square truncation error decreases monotonically to zero with the number of terms in the expansion, at a rate that depends on the covariance function $C(\bx,\bx')$.

If the eigenpairs $(\theta_j,\psi_j)$ in the KL expansion are not available, we need to approximate them using a suitable numerical scheme. One possibility is to perform collocation on~\eqref{eq:inteq}, i.e., solve
\begin{equation}\label{eq:inteqcollocation}
\int_D C(\bx_k,\bx') \, \psi_j(\bx') \, \mathrm{d}\bx'
=
\theta_j \, \psi_j(\bx_k)
,
\quad
k = 1,2,\ldots,M
,
\end{equation}
in some well-chosen integration points $\bx_k$.
In the Nystr\"om method~\cite{atkinson2009numerical}, the integral in equation \eqref{eq:inteqcollocation} is approximated by a suitable numerical integration scheme which uses the collocation points as quadrature nodes:
\begin{equation}\label{eq:discrinteq}
\sum_{q=1}^M w_q \, C(\bx_k,\bx_q) \, \tilde{\psi}_j(\bx_q)
=
\tilde{\theta}_j \tilde{\psi}_j(\bx_k)
,
\quad
k = 1,2,\ldots,M
.
\end{equation}
In matrix notation,
\begin{equation}\label{eq:matrixequation}
\Sigma W\tilde{\Psi}_j = \tilde{\theta}_j\tilde{\Psi}_j,
\end{equation}
where $\Sigma$ is a symmetric positive semi-definite matrix with entries $\Sigma_{k,q}=C(\bx_k,\bx_q)$, $W$ is a diagonal matrix with the weights $w_q$ on the diagonal and $\tilde{\Psi}_j$ is a vector with entries $\tilde{\Psi}_{j,q} = \tilde{\psi}_j(\bx_q)$. In general, $\Sigma W$ is a dense and nonsymmetric matrix. However, the matrix eigenvalue problem~\eqref{eq:matrixequation} can be reformulated into an equivalent matrix eigenvalue problem
\begin{equation}
B \tilde{\Psi}_j^* = \tilde{\theta}_j\tilde{\Psi}_j^*, \nonumber
\end{equation}
where $\tilde{\Psi}_j^*=\sqrt{W}\tilde{\Psi}_j$, and $B=\sqrt{W}\Sigma\sqrt{W}$ is symmetric positive semi-definite. Hence, the eigenvalues $\tilde{\theta}_j$ are real and nonnegative and the eigenvectors $\tilde{\Psi}_j^*$ are orthogonal to each other.

Using~\eqref{eq:discrinteq}, we obtain the so-called Nystr\"om interpolation formula for the eigenfunctions $\tilde{\psi}_j(\bx)$, i.e.,
\begin{equation*}
\tilde{\psi}_j(\bx)
=
\frac{1}{\tilde{\theta}_j} \sum_{q=1}^{M} \sqrt{w_q} \, \tilde{\Psi}_{j,q}^* \, C(\bx,\bx_q)
,
\end{equation*}
with $\tilde{\Psi}_{j,q}^*$ the $q$-th element of the eigenvector $\tilde{\Psi}_j^*$. After a suitable normalization, these eigenvalues and eigenfunctions can be used as an approximate eigenpair in the KL expansion from equation~\eqref{eq:truncatedKL}.

\newcounter{i}
\newcounter{j}
\def\n{3}
\def\nCases{9}
\pgfmathsetmacro{\nPlusOne}{int(\n+1)}
\pgfmathsetmacro{\nCasesPlusOne}{int(\nCases+1)}
\newcommand{\border}{1mm}
\begin{figure}[t]
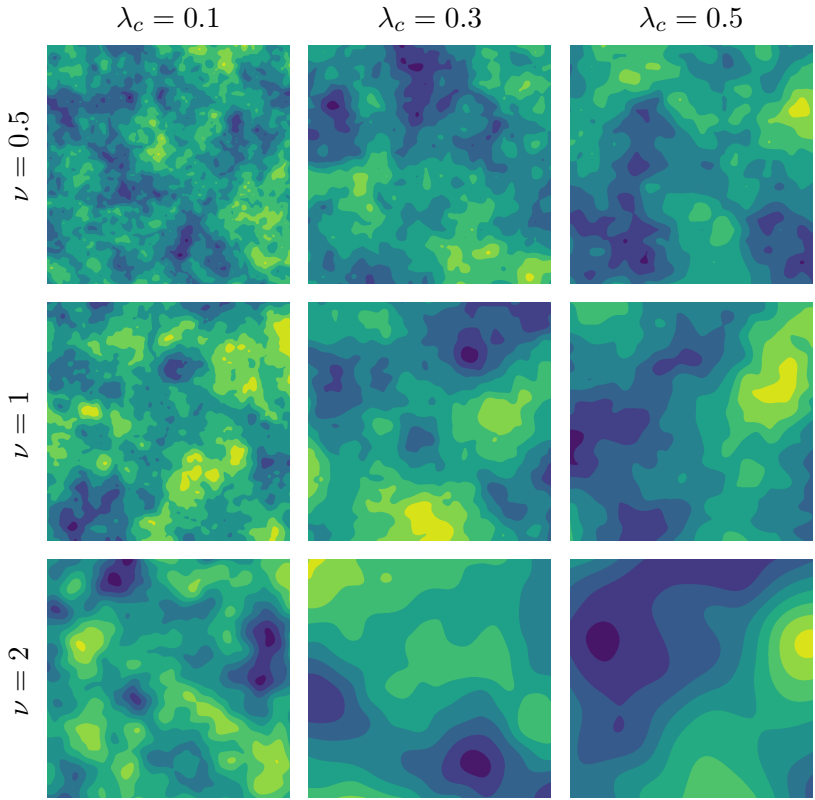

	\centering
	\setlength\tabcolsep{0pt}
	\renewcommand{\arraystretch}{1}
	\begin{tabular}{c@{\hskip 0pt}cccc}&%
	\forloop{j}{1}{\value{j} < \nPlusOne}{
		\pgfmathsetmacro{\idx}{int(\arabic{j}-1)}
		$\lambda_c = \pgfplotstablegetelem{\idx}{0}\of\params\pgfmathprintnumber[fixed]{\pgfplotsretval}$ &
	}\\[\border]
	\forloop{i}{1}{\value{i} < \n}{
		\pgfmathsetmacro{\idx}{int((\arabic{i}-1)*\n)}
		\rotatebox[origin=l]{90}{\hspace{0.06\textwidth}$\nu = \pgfplotstablegetelem{\idx}{1}\of\params\pgfmathprintnumber[fixed]{\pgfplotsretval}$} &
		\forloop{j}{1}{\value{j} < \nPlusOne}{
			\pgfmathsetmacro{\scndidx}{int((\arabic{i}-1)*\n+\arabic{j})}
			\includegraphics[width=0.2\textwidth]{figures/grf_\scndidx.png} &
		}
		\\[\border]
	}
	\pgfmathsetmacro{\idx}{int((\n-1)*\n)}
	\rotatebox[origin=l]{90}{\hspace{0.06\textwidth}$\nu = \pgfplotstablegetelem{\idx}{1}\of\params\pgfmathprintnumber[fixed]{\pgfplotsretval}$} &
	\forloop{j}{1}{\value{j} < \nPlusOne}{
		\pgfmathsetmacro{\scndidx}{int((\n-1)*\n+\arabic{j})}
		\includegraphics[width=0.2\textwidth]{figures/grf_\scndidx.png} &
	}
	\end{tabular}
	\caption{\label{fig:GRFs}Typical realizations of some Gaussian random fields for various combinations of the correlation length $\lambda_c$ and smoothness $\nu$ in Mat\'ern covariance function~\eqref{eq:matern}. These realizations are computed using the truncated KL expansion as outlined in \secref{sec:KL}, where the number of terms in the expansion is chosen according to~\tabref{tab:params}.}
\end{figure}

\section{Numerical Results}\label{sec:numerical_results}

In this section we perform numerical experiments on problem~\eqref{eq:PDE} that illustrate the power of the proposed algorithm. 

\subsection*{Discretization}
Consider the model elliptic PDE from~\eqref{eq:PDE} defined on a two-dimensional computational domain $D=[0,1]^2$, with source term $f(\bx)=1$ and homogenous boundary conditions of Dirichlet-type. We discretize the equation using finite differences on a hierarchy of regular grids with nodes
\begin{equation}
(x^{(i)}_1,x^{(j)}_2)= \left( \frac{i}{m_\ell},  \frac{j}{m_\ell} \right), \quad i,j = 0,\ldots,m_\ell, \nonumber
\end{equation}
where $m_\ell = 1/h_\ell$ and $h_\ell$ satisfies~\eqref{eq:grids} with $h_0=1/4$ and $\varrho=2$.

\subsection*{Modeling of uncertain parameters}
The diffusion coefficient $a(\bx,\boldsymbol{y})$ is modeled as a lognormal random field, i.e.,  $a(\bx,\boldsymbol{y})=\exp(z(\bx,\boldsymbol{y}))$, with $\bar{z}(\bx)=0$, and $C(\bx,\bx')$ is the Mat\'ern covariance function with correlation length $\lambda_c$ and smoothness $\nu$, defined as
\begin{equation}\label{eq:matern}
C(\bx,\bx') = \frac{2^{1-\nu}}{\Gamma(\nu)}\left(\sqrt{2\nu}{\rho}\right)^\nu K_\nu\left(\sqrt{2\nu}{\rho}\right), \quad \rho=\frac{\|\bx-\bx'\|_2}{\lambda_c}.
\end{equation}
In the following, we consider the parameter sets $\lambda_c=\{\pgfplotstablegetelem{0}{0}\of\params\pgfmathprintnumber[fixed]{\pgfplotsretval},\pgfplotstablegetelem{1}{0}\of\params\pgfmathprintnumber[fixed]{\pgfplotsretval},\pgfplotstablegetelem{2}{0}\of\params\pgfmathprintnumber[fixed]{\pgfplotsretval}\}$ and $\nu=\{\pgfplotstablegetelem{0}{1}\of\params\pgfmathprintnumber[fixed]{\pgfplotsretval},\pgfplotstablegetelem{4}{1}\of\params\pgfmathprintnumber[fixed]{\pgfplotsretval},\pgfplotstablegetelem{7}{1}\of\params\pgfmathprintnumber[fixed]{\pgfplotsretval}\}$.
A typical realization of the random field $z(\bx,\boldsymbol{y})$ for all parameter combinations of correlation length $\lambda_c$ and smoothness $\nu$ is shown in~\figref{fig:GRFs}.
Each of the random fields is represented using the truncated KL expansion in equation~\eqref{eq:truncatedKL}, where the number of terms in the expansion is chosen according to~\tabref{tab:params}. In particular, we used the identity
\begin{equation*}
\sum_{j=1}^{\infty} \theta_j
=
\int_D \mathrm{d}\bx = 1
\end{equation*}
to ensure that at least 99.8\% of the variability in the random field is captured by the first $\stochd$ terms, see~\cite{ernst2009efficient}.
The eigenvalues and eigenfunctions are approximated numerically using the Nystr\"om method outlined in~\secref{sec:KL}.

\begin{table}[t]
\centering
$\begin{array}{lcccc}\toprule
	&%
	\forloop{j}{1}{\value{j} < \nPlusOne}{
		\pgfmathsetmacro{\idx}{int(\arabic{j}-1)}
		\lambda_c = \pgfplotstablegetelem{\idx}{0}\of\params\pgfmathprintnumber[fixed]{\pgfplotsretval} &
	}\tabularnewline\midrule
	\forloop{i}{1}{\value{i} < \n}{
		\pgfmathsetmacro{\idx}{int((\arabic{i}-1)*\n)}
		\nu = \pgfplotstablegetelem{\idx}{1}\of\params\pgfmathprintnumber[fixed]{\pgfplotsretval} &
		\forloop{j}{1}{\value{j} < \nPlusOne}{
			\pgfmathsetmacro{\scndidx}{int((\arabic{i}-1)*\n+\arabic{j}-1)}
			\pgfplotstablegetelem{\scndidx}{2}\of\params\pgfmathprintnumber[fixed,1000 sep={}]{\pgfplotsretval} &
		}\\
	}
	\pgfmathsetmacro{\idx}{int((\arabic{i}-1)*\n)}
	\nu = \pgfplotstablegetelem{\idx}{1}\of\params\pgfmathprintnumber[fixed]{\pgfplotsretval} &
	\forloop{j}{1}{\value{j} < \nPlusOne}{
		\pgfmathsetmacro{\scndidx}{int((\arabic{i}-1)*\n+\arabic{j}-1)}
		\pgfplotstablegetelem{\scndidx}{2}\of\params\pgfmathprintnumber[fixed,1000 sep={}]{\pgfplotsretval} &
	}\\\bottomrule
\end{array}$
\caption{\label{tab:params}Number of terms in the KL expansion used to represent the different Gaussian random fields in the numerical experiments. As the smoothness $\nu$ in the Mat\'ern covariance increases, fewer terms are needed to accurately model the random fields using a KL expansion.}
\end{table}

\figref{fig:eig_decay} shows the decay of the eigenvalues $\tilde{\theta}_j$  for each parameter combination. Note that as the smoothness $\nu$ increases, the eigenvalues decay faster and fewer terms are required in the expansion. For decreasing correlation lengths $\lambda_c$, there is a growing pre-asymptotic phase where the KL eigenvalues do not decay. Note also that there is an end effect for the smallest eigenvalues, where the asymptotic decay rate is not preserved. Similar numerical results in case of a circulant embedding method and corresponding analysis can be found in~\cite{graham2017analysis}.

\subsection*{Quantity of interest}
As quantity of interest, we choose the value of $u(\bx,\cdot)$ at $\bx=(1/2,1/2)$.

\subsection*{Deterministic solver}
Every sample of~\eqref{eq:PDE} is solved using the FMG approach outlined in~\secref{sec:intro_det}. We used symmetric Gauss-Seidel (SGS)~\cite{saad2003iterative} as a smoother. The actual time needed to compute a solution on each level was used as an estimate for the cost $C_{\ell}$.

\subsection*{QMC method}
We used a randomly shifted rank-1 lattice rule with a generating vector computed using the fast CBC construction from the QMC4PDE website~\cite{QMC4PDE}. We used a total of $\nbshifts=20$ independent random shifts. Note that, since the total number of points in the lattice rule is not fixed a priori, the lattice rule is actually a lattice sequence, see~\cite{dick2013high} for details.

\begin{figure}
\centering
\setlength\figurewidth{0.29\textwidth}
\setlength\figureheight{0.29\textwidth}
\forloop{i}{1}{\value{i} < \n}{
	\ifthenelse{\arabic{i}>1}{
	\tikzexternalenable
	\tikzsetnextfilename{eigenvalues_\arabic{i}}%
	\begin{tikzpicture}[trim axis left,trim axis right]

\pgfmathsetmacro{\first}{int((\arabic{i}-1)*\n)}
\pgfmathsetmacro{\firstPlusOne}{int((\arabic{i}-1)*\n+1)}
\pgfmathsetmacro{\scnd}{int((\arabic{i}-1)*\n+1)}
\pgfmathsetmacro{\scndPlusOne}{int((\arabic{i}-1)*\n+2)}
\pgfmathsetmacro{\third}{int((\arabic{i}-1)*\n+2)}
\pgfmathsetmacro{\thirdPlusOne}{int((\arabic{i}-1)*\n+3)}

\begin{axis}[
width=\figurewidth,
height=\figureheight,
scale only axis,
xmode=log,
xmin=1e0,
xmax=1e4,
xminorticks=true,
xlabel={\small $\klcntr$},
every outer x axis line/.append style={white!15!black},
every x tick label/.append style={font=\color{white!15!black}\scriptsize},
every x label/.append style={font=\color{white!15!black}\scriptsize},
ymode=log,
ymin=1e-6,
ymax=1e0,
yminorticks=true,
ylabel={\small \ifthenelse{\arabic{i}>1}{}{$\theta_r$}},
every outer y axis line/.append style={white!15!black},
every y tick label/.append style={font=\color{white!15!black}\scriptsize},
every y label/.append style={font=\color{white!15!black}\scriptsize},
every y tick label/.append style={opacity=0},
axis background/.style={fill=white},
legend style={legend cell align=left,align=left,draw=none,font=\scriptsize,at={(0.97,0.97)},anchor=north east,fill=none},
title={$\nu = \pgfplotstablegetelem{\first}{1}\of\params\pgfmathprintnumber[fixed]{\pgfplotsretval}$}
]

\addplot [LineStyle1]
  table[y expr={\thisrowno{1}^2}]{plot_data/eigenvalues_\firstPlusOne.txt};
\addlegendentry{$\lambda_c=\pgfplotstablegetelem{\first}{0}\of\params\pgfmathprintnumber[fixed]{\pgfplotsretval}$};

\addplot [LineStyle2]
  table[y expr={\thisrowno{1}^2}]{plot_data/eigenvalues_\scndPlusOne.txt};
\addlegendentry{$\lambda_c=\pgfplotstablegetelem{\scnd}{0}\of\params\pgfmathprintnumber[fixed]{\pgfplotsretval}$};

\addplot [LineStyle3]
  table[y expr={\thisrowno{1}^2}]{plot_data/eigenvalues_\thirdPlusOne.txt};
\addlegendentry{$\lambda_c=\pgfplotstablegetelem{\third}{0}\of\params\pgfmathprintnumber[fixed]{\pgfplotsretval}$};

\end{axis}
\end{tikzpicture}%
	\tikzexternaldisable
\hspace{0.25cm}\hfill%
	}{
	\tikzexternalenable
	\tikzsetnextfilename{eigenvalues_\arabic{i}}%
	\begin{tikzpicture}[trim axis left,trim axis right]

\pgfmathsetmacro{\first}{int((\arabic{i}-1)*\n)}
\pgfmathsetmacro{\firstPlusOne}{int((\arabic{i}-1)*\n+1)}
\pgfmathsetmacro{\scnd}{int((\arabic{i}-1)*\n+1)}
\pgfmathsetmacro{\scndPlusOne}{int((\arabic{i}-1)*\n+2)}
\pgfmathsetmacro{\third}{int((\arabic{i}-1)*\n+2)}
\pgfmathsetmacro{\thirdPlusOne}{int((\arabic{i}-1)*\n+3)}

\begin{axis}[
width=\figurewidth,
height=\figureheight,
scale only axis,
xmode=log,
xmin=1e0,
xmax=1e4,
xminorticks=true,
xlabel={\small $\klcntr$},
every outer x axis line/.append style={white!15!black},
every x tick label/.append style={font=\color{white!15!black}\scriptsize},
every x label/.append style={font=\color{white!15!black}\scriptsize},
ymode=log,
ymin=1e-6,
ymax=1e0,
yminorticks=true,
ylabel={\small \ifthenelse{\arabic{i}>1}{}{$\Eigenvalue$}},
every outer y axis line/.append style={white!15!black},
every y tick label/.append style={font=\color{white!15!black}\scriptsize},
every y label/.append style={font=\color{white!15!black}\scriptsize},
axis background/.style={fill=white},
legend style={legend cell align=left,align=left,draw=none,font=\scriptsize,at={(0.97,0.97)},anchor=north east,fill=none},
title={$\nu = \pgfplotstablegetelem{\first}{1}\of\params\pgfmathprintnumber[fixed]{\pgfplotsretval}$}
]

\addplot [LineStyle1]
  table[y expr={\thisrowno{1}^2}]{plot_data/eigenvalues_\firstPlusOne.txt};
\addlegendentry{$\lambda_c=\pgfplotstablegetelem{\first}{0}\of\params\pgfmathprintnumber[fixed]{\pgfplotsretval}$};

\addplot [LineStyle2]
  table[y expr={\thisrowno{1}^2}]{plot_data/eigenvalues_\scndPlusOne.txt};
\addlegendentry{$\lambda_c=\pgfplotstablegetelem{\scnd}{0}\of\params\pgfmathprintnumber[fixed]{\pgfplotsretval}$};

\addplot [LineStyle3]
  table[y expr={\thisrowno{1}^2}]{plot_data/eigenvalues_\thirdPlusOne.txt};
\addlegendentry{$\lambda_c=\pgfplotstablegetelem{\third}{0}\of\params\pgfmathprintnumber[fixed]{\pgfplotsretval}$};

\end{axis}
\end{tikzpicture}%
	\tikzexternaldisable
\hfill%
	}%
}%
	\tikzexternalenable
	\tikzsetnextfilename{eigenvalues_\arabic{i}}%
	\begin{tikzpicture}[trim axis left,trim axis right]

\pgfmathsetmacro{\first}{int((\arabic{i}-1)*\n)}
\pgfmathsetmacro{\firstPlusOne}{int((\arabic{i}-1)*\n+1)}
\pgfmathsetmacro{\scnd}{int((\arabic{i}-1)*\n+1)}
\pgfmathsetmacro{\scndPlusOne}{int((\arabic{i}-1)*\n+2)}
\pgfmathsetmacro{\third}{int((\arabic{i}-1)*\n+2)}
\pgfmathsetmacro{\thirdPlusOne}{int((\arabic{i}-1)*\n+3)}

\begin{axis}[
width=\figurewidth,
height=\figureheight,
scale only axis,
xmode=log,
xmin=1e0,
xmax=1e4,
xminorticks=true,
xlabel={\small $\klcntr$},
every outer x axis line/.append style={white!15!black},
every x tick label/.append style={font=\color{white!15!black}\scriptsize},
every x label/.append style={font=\color{white!15!black}\scriptsize},
ymode=log,
ymin=1e-6,
ymax=1e0,
yminorticks=true,
ylabel={\small \ifthenelse{\arabic{i}>1}{}{$\theta_r$}},
every outer y axis line/.append style={white!15!black},
every y tick label/.append style={font=\color{white!15!black}\scriptsize},
every y label/.append style={font=\color{white!15!black}\scriptsize},
every y tick label/.append style={opacity=0},
axis background/.style={fill=white},
legend style={legend cell align=left,align=left,draw=none,font=\scriptsize,at={(0.97,0.97)},anchor=north east,fill=none},
title={$\nu = \pgfplotstablegetelem{\first}{1}\of\params\pgfmathprintnumber[fixed]{\pgfplotsretval}$}
]

\addplot [LineStyle1]
  table[y expr={\thisrowno{1}^2}]{plot_data/eigenvalues_\firstPlusOne.txt};
\addlegendentry{$\lambda_c=\pgfplotstablegetelem{\first}{0}\of\params\pgfmathprintnumber[fixed]{\pgfplotsretval}$};

\addplot [LineStyle2]
  table[y expr={\thisrowno{1}^2}]{plot_data/eigenvalues_\scndPlusOne.txt};
\addlegendentry{$\lambda_c=\pgfplotstablegetelem{\scnd}{0}\of\params\pgfmathprintnumber[fixed]{\pgfplotsretval}$};

\addplot [LineStyle3]
  table[y expr={\thisrowno{1}^2}]{plot_data/eigenvalues_\thirdPlusOne.txt};
\addlegendentry{$\lambda_c=\pgfplotstablegetelem{\third}{0}\of\params\pgfmathprintnumber[fixed]{\pgfplotsretval}$};

\end{axis}
\end{tikzpicture}%
	\tikzexternaldisable

\caption{\label{fig:eig_decay}Decay of the eigenvalues in the KL expansion from \secref{sec:KL} for various combinations of the correlation length $\lambda_c$ and smoothness $\nu$ in the Mat\'ern covariance function.}
\end{figure}
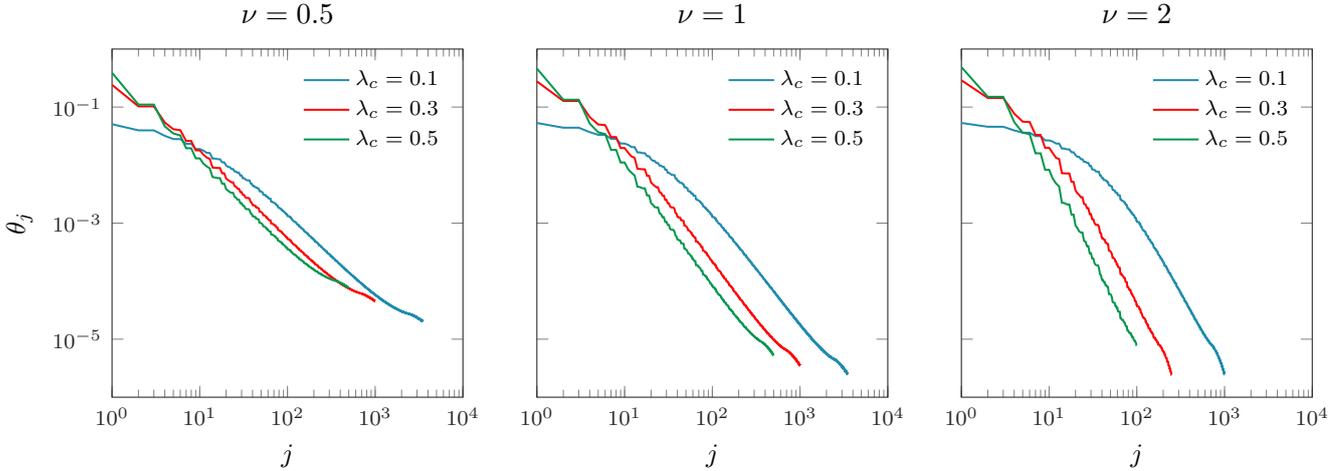

\subsection*{Computing environment}
All our simulations are performed on a workstation with 20 Intel Xeon E5645 processors running at 2.4 GHz and using a total of 48~GB of RAM. All samples are taken in parallel. Since the Monte Carlo sampling is embarrassingly parallel, we observed almost linear speedup with respect to the number of cores. For the QMC-based methods, we ran an independent estimator on each core, where the number of shifts equaled the number of computational cores. For details on work scheduling of MLMC samplers on parallel machines, we refer to~\cite{drzisga2017scheduling}.

The MG-MLQMC method is implemented within the framework of the Julia package {\texttt{Multil-} \texttt{evelEstimators.jl}}, see~\cite{robbe2018multilevelestimators}, and we used this software in all our experiments.

\subsection{Numerical rate determination}

We begin our numerical tests by determining the rates $\alpha$, $\beta$, $\gamma$ and $\lambda$ in assumptions \eqref{A1}--\eqref{A4}. 

\figref{fig:E} and \figref{fig:V} are standard figures used in MLMC literature. They plot the behavior of expected value and variance of the quantity of interest $F_{\ell}$ and the difference $F_{\ell}-F_{\ell-1}$ as a function of the level parameter $\ell$, for all parameter combinations listed in~\tabref{tab:params}. Observe that the estimated rate $\alpha$ in assumption~\eqref{A1} is approximately constant. The estimated rate $\beta$ in assumption~\eqref{A2} increases as the smoothness $\nu$ and correlation length $\lambda_c$ in the Mat\'ern covariance function increase. In addition, note that for $\lambda_c=0.1$, the graphs of $\text{V}[{F_{\ell}}]$ and $\text{V}[{F_{\ell}-F_{\ell-1}}]$ intersect between $\ell=1$ and $\ell=2$, indicating that the coarse level uses too few degrees of freedom. In this case, there is no benefit in including the coarsest mesh into the hierarchy. However, for the sake of uniformity, we will use the same mesh hierarchy for all test cases.

\figref{fig:MSE} shows the variance of the QMC estimator for the expected value of the differences $F_{\ell}-F_{\ell-1}$ as the number of samples $N_\ell$ is increased. Note that the numerically observed rate $\lambda$ from assumption~\eqref{A4} seems to decrease (that is, the lattice rule performs better) as the correlation length $\lambda_c$ increases. However, from our analysis in~\secref{sec:cost_anal}, and specifically Proposition~\ref{prop:cost_reduction}, it follows that this effect is negligible compared to the increase of the rate $\beta$. Observe also the variance reduction from level to level, i.e., the offset between the lines.

Finally, the rate $\gamma$ in assumption~\eqref{A3} is fixed by the FMG solver. We found numerically that $\gamma\approx 2.12$.

Based on the analysis from~\secref{sec:cost_anal}, we thus expect that the sample recycling is less efficient when
\begin{enumerate}
\item[(i)] the smoothness $\nu$ of the covariance function increases (i.e., moving from top to bottom in~\tabref{tab:params}), and 
\item[(ii)] the correlation length $\lambda_c$ of the covariance function increases (i.e., moving from left to right in~\tabref{tab:params}).
\end{enumerate}

\begin{figure}[p]
\centering
\setlength\figurewidth{0.3\textwidth}
\setlength\figureheight{0.3\textwidth}
\foreach \j in {1,...,\n}{
	\foreach \i in {1,...,\n}{
		\pgfmathsetmacro{\case}{int((\j-1)*\n+\i)}
		\ifthenelse{\i>1}{%
			\ifthenelse{\j<3}{%
	\tikzexternalenable
	\tikzsetnextfilename{plot_E_\case}%
	\begin{tikzpicture}[trim axis left,trim axis right]

\pgfmathsetmacro{\caseMinusOne}{int(\case-1)}

\pgfplotstableread[header=false]{plot_data/table_rates_E.txt}\rates

\begin{axis}[
width=\figurewidth,
height=\figureheight,
scale only axis,
xmin=0,
xmax=5,
xminorticks=true,
xtick={0,1,...,5},
every outer x axis line/.append style={white!15!black},
every x tick label/.append style={font=\color{white!15!black}\scriptsize},
every x label/.append style={font=\color{white!15!black}\scriptsize},
ymin=-20,
ymax=0,
yminorticks=true,
every y tick label/.append style={opacity=0},
every x tick label/.append style={opacity=0},
every outer y axis line/.append style={white!15!black},
every y tick label/.append style={font=\color{white!15!black}\scriptsize},
every y label/.append style={font=\color{white!15!black}\scriptsize},
axis background/.style={fill=white},
legend style={legend cell align=left,align=left,draw=none,font=\scriptsize,at={(0.03,0.03)},anchor=south west,fill=none}
]

\addplot [DottedLineStyle2]
  table[]{plot_data/case_\case_rates_E.txt};

\addplot [DashedDottedLineStyle2]
  table[]{plot_data/case_\case_rates_dE.txt};

\node[anchor=south west] at (rel axis cs:0.03,0.03) {\small $\alpha=\pgfplotstablegetelem{\caseMinusOne}{0}\of\rates\pgfmathprintnumber[fixed]{\pgfplotsretval}$};

\end{axis}
\end{tikzpicture}%
	\tikzexternaldisable
\hfill%
			}{%
	\tikzexternalenable
	\tikzsetnextfilename{plot_E_\case}%
	\begin{tikzpicture}[trim axis left,trim axis right]

\pgfmathsetmacro{\caseMinusOne}{int(\case-1)}

\pgfplotstableread[header=false]{plot_data/table_rates_E.txt}\rates

\begin{axis}[
width=\figurewidth,
height=\figureheight,
scale only axis,
xmin=0,
xmax=5,
xminorticks=true,
xlabel={\small $\levelcntr$},
xtick={0,1,...,5},
every outer x axis line/.append style={white!15!black},
every x tick label/.append style={font=\color{white!15!black}\scriptsize},
every x label/.append style={font=\color{white!15!black}\scriptsize},
ymin=-20,
ymax=0,
yminorticks=true,
every y tick label/.append style={opacity=0},
every outer y axis line/.append style={white!15!black},
every y tick label/.append style={font=\color{white!15!black}\scriptsize},
every y label/.append style={font=\color{white!15!black}\scriptsize},
axis background/.style={fill=white},
legend style={legend cell align=left,align=left,draw=none,font=\scriptsize,at={(0.03,0.03)},anchor=south west,fill=none}
]

\addplot [DottedLineStyle2]
  table[]{plot_data/case_\case_rates_E.txt};

\addplot [DashedDottedLineStyle2]
  table[]{plot_data/case_\case_rates_dE.txt};

\node[anchor=south west] at (rel axis cs:0.03,0.03) {\small $\alpha=\pgfplotstablegetelem{\caseMinusOne}{0}\of\rates\pgfmathprintnumber[fixed]{\pgfplotsretval}$};

\end{axis}
\end{tikzpicture}%
	\tikzexternaldisable
\hfill%
			}
		}{%
			\ifthenelse{\j<3}{%
	\tikzexternalenable
	\tikzsetnextfilename{plot_E_\case}%
	\begin{tikzpicture}[trim axis left,trim axis right]

\pgfmathsetmacro{\caseMinusOne}{int(\case-1)}

\pgfplotstableread[header=false]{plot_data/table_rates_E.txt}\rates

\begin{axis}[
width=\figurewidth,
height=\figureheight,
scale only axis,
xmin=0,
xmax=5,
xminorticks=true,
xtick={0,1,...,5},
every outer x axis line/.append style={white!15!black},
every x tick label/.append style={font=\color{white!15!black}\scriptsize},
every x label/.append style={font=\color{white!15!black}\scriptsize},
ymin=-20,
ymax=0,
yminorticks=true,
ylabel={\small $\log_2(|\ExpOf{\;\cdot\;}|)$},
every x tick label/.append style={opacity=0},
every outer y axis line/.append style={white!15!black},
every y tick label/.append style={font=\color{white!15!black}\scriptsize},
every y label/.append style={font=\color{white!15!black}\scriptsize},
axis background/.style={fill=white},
legend style={legend cell align=left,align=left,draw=none,font=\scriptsize,at={(0.03,0.03)},anchor=south west,fill=none}
]

\addplot [DottedLineStyle2]
  table[]{plot_data/case_\case_rates_E.txt};

\addplot [DashedDottedLineStyle2]
  table[]{plot_data/case_\case_rates_dE.txt};

\node[anchor=south west] at (rel axis cs:0.03,0.03) {\small $\alpha=\pgfplotstablegetelem{\caseMinusOne}{0}\of\rates\pgfmathprintnumber[fixed,zerofill]{\pgfplotsretval}$};

\end{axis}
\end{tikzpicture}%
	\tikzexternaldisable
\hfill%
			}{%
	\tikzexternalenable
	\tikzsetnextfilename{plot_E_\case}%
	\begin{tikzpicture}[trim axis left,trim axis right]

\pgfmathsetmacro{\caseMinusOne}{int(\case-1)}

\pgfplotstableread[header=false]{plot_data/table_rates_E.txt}\rates

\begin{axis}[
width=\figurewidth,
height=\figureheight,
scale only axis,
xmin=0,
xmax=5,
xminorticks=true,
xlabel={\small $\levelcntr$},
xtick={0,1,...,5},
every outer x axis line/.append style={white!15!black},
every x tick label/.append style={font=\color{white!15!black}\scriptsize},
every x label/.append style={font=\color{white!15!black}\scriptsize},
ymin=-20,
ymax=0,
yminorticks=true,
ylabel={\small $\log_2(|\ExpOf{\;\cdot\;}|)$},
every outer y axis line/.append style={white!15!black},
every y tick label/.append style={font=\color{white!15!black}\scriptsize},
every y label/.append style={font=\color{white!15!black}\scriptsize},
axis background/.style={fill=white},
legend style={legend cell align=left,align=left,draw=none,font=\scriptsize,at={(0.03,0.03)},anchor=south west,fill=none}
]

\addplot [DottedLineStyle2]
  table[]{plot_data/case_\case_rates_E.txt};
\label{E};

\addplot [DashedDottedLineStyle2]
  table[]{plot_data/case_\case_rates_dE.txt};
\label{dE};

\node[anchor=south west] at (rel axis cs:0.03,0.03) {\small $\alpha=\pgfplotstablegetelem{\caseMinusOne}{0}\of\rates\pgfmathprintnumber[fixed]{\pgfplotsretval}$};

\end{axis}
\end{tikzpicture}%
	\tikzexternaldisable
\hfill%
			}
		}%
	}\\%
}%
\caption{\label{fig:E}Expected values $\text{E}[{F_{\ell}}]$ (\ref{E}) and $\text{E}[{F_{\ell}-F_{\ell-1}}]$ (\ref{dE}) for $\lambda_c=0.1$ (left column), $\lambda_c=0.3$ (middle column), $\lambda_c=0.5$ (right column) and $\nu=0.5$ (top row), $\nu=1$ (middle row), $\nu=2$ (bottom row). The estimated rate $\alpha$ corresponds to the rate in assumption~\eqref{A1}.}
\end{figure}
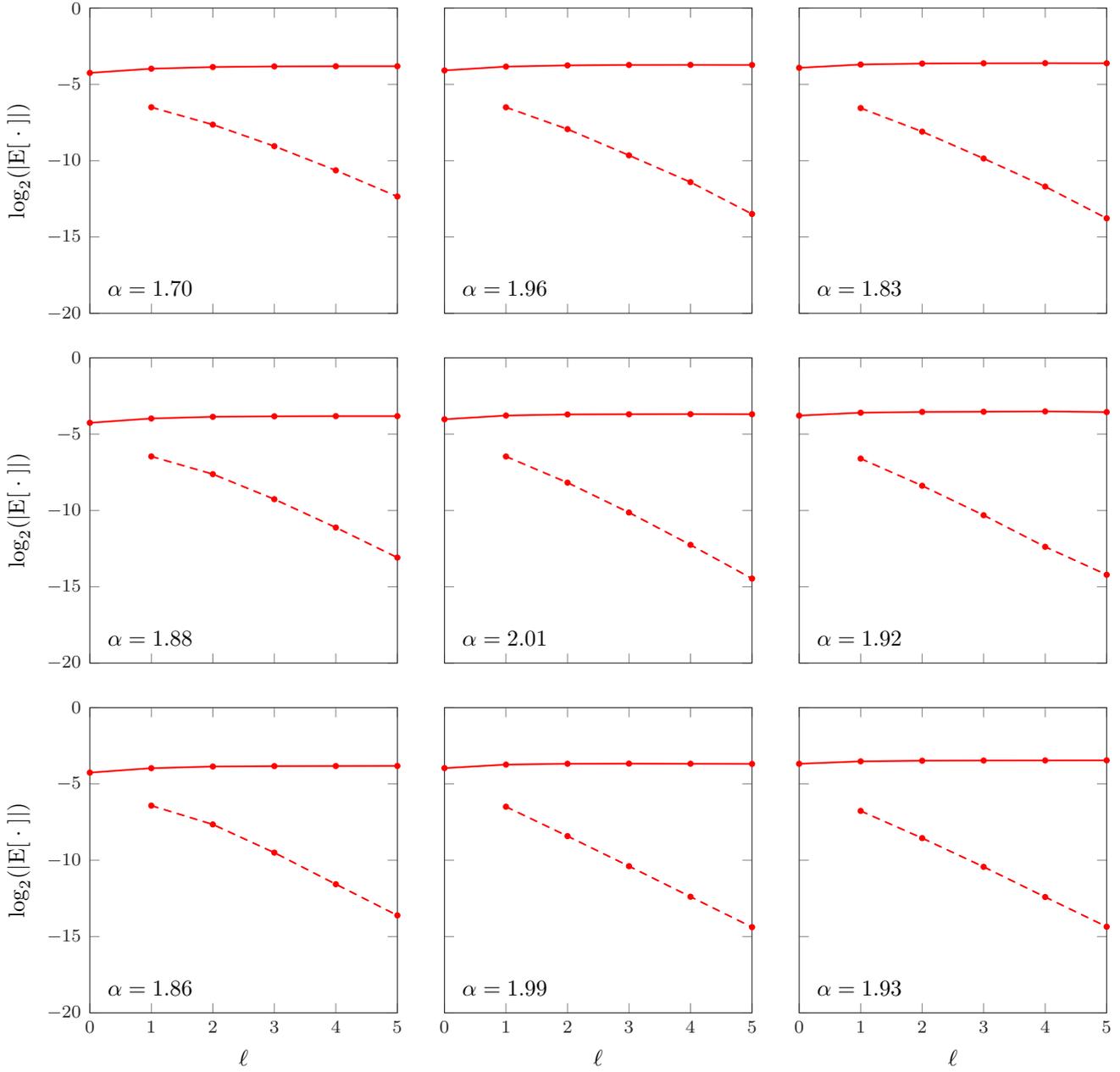

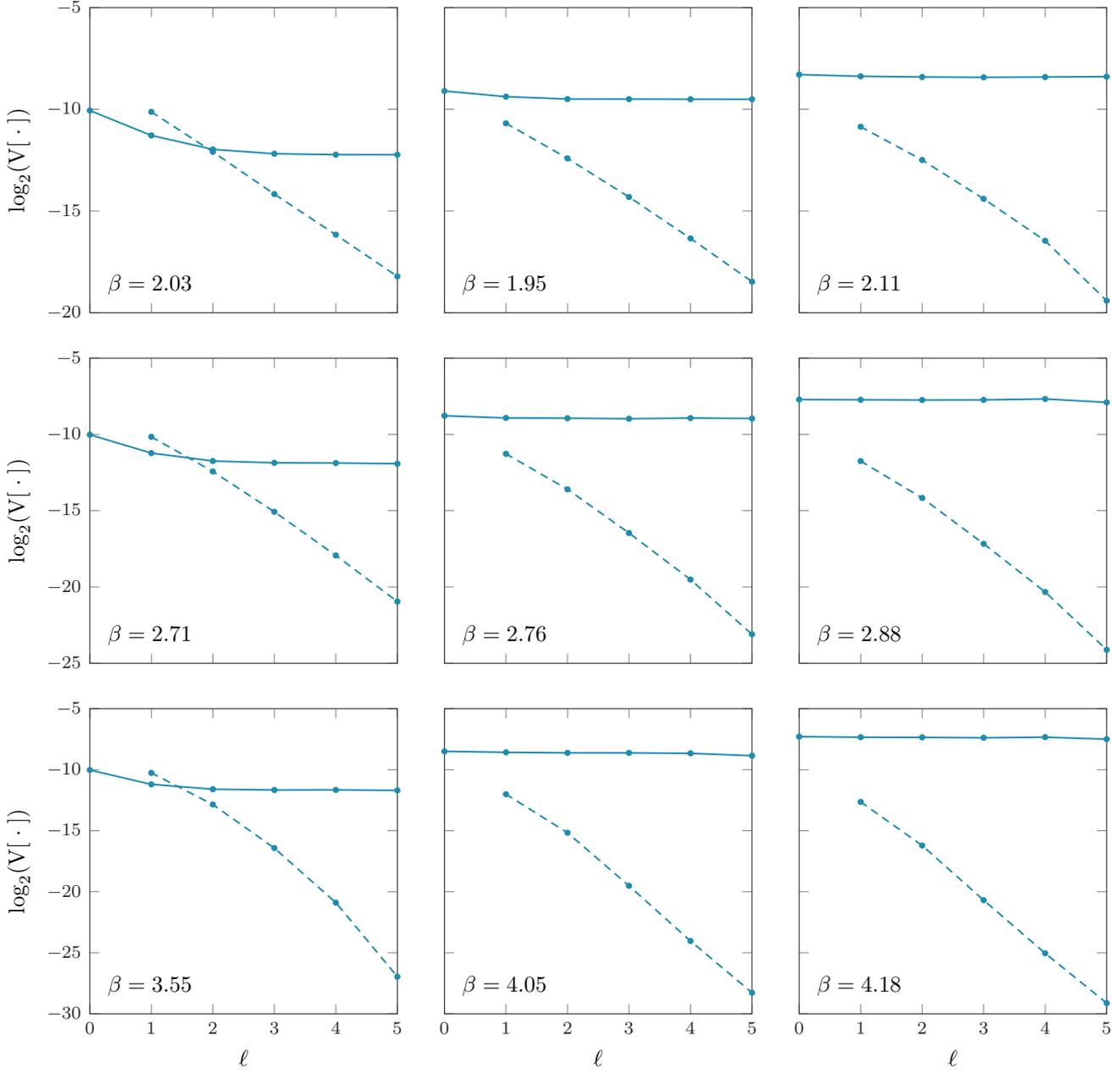
\begin{figure}[p]
\flushleft
\setlength\figurewidth{0.3\textwidth}
\setlength\figureheight{0.3\textwidth}
\foreach \j in {1,...,\n}{
	\foreach \i in {1,...,\n}{
		\pgfmathsetmacro{\case}{int((\j-1)*\n+\i)}
		\ifthenelse{\i>1}{%
			\ifthenelse{\j<3}{%
	\tikzexternalenable
	\tikzsetnextfilename{plot_V_\case}%
	\begin{tikzpicture}[trim axis left,trim axis right]

\pgfmathsetmacro{\caseMinusOne}{int(\case-1)}
\pgfmathsetmacro{\MyYMin}{(\case<(\n+1))?-20:((\case<(2*\n+1))?-25:-35)}
\pgfplotstableread[header=false]{plot_data/table_rates_V.txt}\rates

\begin{axis}[
width=\figurewidth,
height=\figureheight,
scale only axis,
xmin=0,
xmax=5,
xminorticks=true,
xtick={0,1,...,5},
every outer x axis line/.append style={white!15!black},
every x tick label/.append style={font=\color{white!15!black}\scriptsize},
every x label/.append style={font=\color{white!15!black}\scriptsize},
ymin=\MyYMin,
ymax=-5,
yminorticks=true,
every y tick label/.append style={opacity=0},
every x tick label/.append style={opacity=0},
every outer y axis line/.append style={white!15!black},
every y tick label/.append style={font=\color{white!15!black}\scriptsize},
every y label/.append style={font=\color{white!15!black}\scriptsize},
axis background/.style={fill=white},
legend style={legend cell align=left,align=left,draw=none,font=\scriptsize,at={(0.03,0.03)},anchor=south west,fill=none}
]

\addplot [DottedLineStyle1]
  table[]{plot_data/case_\case_rates_V.txt};

\addplot [DashedDottedLineStyle1]
  table[]{plot_data/case_\case_rates_dV.txt};

\node[anchor=south west] at (rel axis cs:0.03,0.03) {\small $\beta=\pgfplotstablegetelem{\caseMinusOne}{0}\of\rates\pgfmathprintnumber[fixed,zerofill]{\pgfplotsretval}$};

\end{axis}
\end{tikzpicture}%
	\tikzexternaldisable
\hfill%
			}{%
	\tikzexternalenable
	\tikzsetnextfilename{plot_V_\case}%
	\begin{tikzpicture}[trim axis left,trim axis right]

\pgfmathsetmacro{\caseMinusOne}{int(\case-1)}
\pgfmathsetmacro{\MyYMin}{(\case<(\n+1))?-20:(\case<(2*\n+1))?-25:-30}
\pgfplotstableread[header=false]{plot_data/table_rates_V.txt}\rates

\begin{axis}[
width=\figurewidth,
height=\figureheight,
scale only axis,
xmin=0,
xmax=5,
xminorticks=true,
xlabel={\small $\levelcntr$},
xtick={0,1,...,5},
ytick={-35,-30,...,-5},
every outer x axis line/.append style={white!15!black},
every x tick label/.append style={font=\color{white!15!black}\scriptsize},
every x label/.append style={font=\color{white!15!black}\scriptsize},
ymin=\MyYMin,
ymax=-5,
yminorticks=true,
every y tick label/.append style={opacity=0},
every outer y axis line/.append style={white!15!black},
every y tick label/.append style={font=\color{white!15!black}\scriptsize},
every y label/.append style={font=\color{white!15!black}\scriptsize},
axis background/.style={fill=white},
legend style={legend cell align=left,align=left,draw=none,font=\scriptsize,at={(0.03,0.03)},anchor=south west,fill=none}
]

\addplot [DottedLineStyle1]
  table[]{plot_data/case_\case_rates_V.txt};

\addplot [DashedDottedLineStyle1]
  table[]{plot_data/case_\case_rates_dV.txt};

\node[anchor=south west] at (rel axis cs:0.03,0.03) {\small $\beta=\pgfplotstablegetelem{\caseMinusOne}{0}\of\rates\pgfmathprintnumber[fixed,zerofill]{\pgfplotsretval}$};

\end{axis}
\end{tikzpicture}%
	\tikzexternaldisable
\hfill%
			}
		}{%
			\ifthenelse{\j<3}{%
	\tikzexternalenable
	\tikzsetnextfilename{plot_V_\case}%
	\begin{tikzpicture}[trim axis left,trim axis right]

\pgfmathsetmacro{\caseMinusOne}{int(\case-1)}
\pgfmathsetmacro{\MyYMin}{(\case<(\n+1))?-20:((\case<(2*\n+1))?-25:-30)}
\pgfplotstableread[header=false]{plot_data/table_rates_V.txt}\rates

\begin{axis}[
width=\figurewidth,
height=\figureheight,
scale only axis,
xmin=0,
xmax=5,
xminorticks=true,
every x tick label/.append style={opacity=0},
xtick={0,1,...,5},
every outer x axis line/.append style={white!15!black},
every x tick label/.append style={font=\color{white!15!black}\scriptsize},
every x label/.append style={font=\color{white!15!black}\scriptsize},
ymin=\MyYMin,
ymax=-5,
yminorticks=true,
ylabel={\small $\log_2(\VarOf{\;\cdot\;})$},
every outer y axis line/.append style={white!15!black},
every y tick label/.append style={font=\color{white!15!black}\scriptsize},
every y label/.append style={font=\color{white!15!black}\scriptsize},
axis background/.style={fill=white},
legend style={legend cell align=left,align=left,draw=none,font=\scriptsize,at={(0.03,0.03)},anchor=south west,fill=none}
]

\addplot [DottedLineStyle1]
  table[]{plot_data/case_\case_rates_V.txt};

\addplot [DashedDottedLineStyle1]
  table[]{plot_data/case_\case_rates_dV.txt};

\node[anchor=south west] at (rel axis cs:0.03,0.03) {\small $\beta=\pgfplotstablegetelem{\caseMinusOne}{0}\of\rates\pgfmathprintnumber[fixed,zerofill]{\pgfplotsretval}$};

\end{axis}
\end{tikzpicture}%
	\tikzexternaldisable
\hfill%
			}{%
	\tikzexternalenable
	\tikzsetnextfilename{plot_V_\case}%
	\begin{tikzpicture}[trim axis left,trim axis right]

\pgfmathsetmacro{\caseMinusOne}{int(\case-1)}
\pgfmathsetmacro{\MyYMin}{(\case<(\n+1))?-20:(\case<(2*\n+1))?-25:-30}
\pgfplotstableread[header=false]{plot_data/table_rates_V.txt}\rates

\begin{axis}[
width=\figurewidth,
height=\figureheight,
scale only axis,
xmin=0,
xmax=5,
xminorticks=true,
xlabel={\small $\levelcntr$},
xtick={0,1,...,5},
every outer x axis line/.append style={white!15!black},
every x tick label/.append style={font=\color{white!15!black}\scriptsize},
every x label/.append style={font=\color{white!15!black}\scriptsize},
ymin=\MyYMin,
ymax=-5,
ytick={-35,-30,...,-5},
yminorticks=true,
ylabel={\small $\log_2(\VarOf{\;\cdot\;})$},
every outer y axis line/.append style={white!15!black},
every y tick label/.append style={font=\color{white!15!black}\scriptsize},
every y label/.append style={font=\color{white!15!black}\scriptsize},
axis background/.style={fill=white},
legend style={legend cell align=left,align=left,draw=none,font=\scriptsize,at={(0.03,0.03)},anchor=south west,fill=none}
]

\addplot [DottedLineStyle1]
  table[]{plot_data/case_\case_rates_V.txt};
\label{V};				

\addplot [DashedDottedLineStyle1]
  table[]{plot_data/case_\case_rates_dV.txt};
\label{dV};

\node[anchor=south west] at (rel axis cs:0.03,0.03) {\small $\beta=\pgfplotstablegetelem{\caseMinusOne}{0}\of\rates\pgfmathprintnumber[fixed,zerofill]{\pgfplotsretval}$};

\end{axis}
\end{tikzpicture}%
	\tikzexternaldisable
\hfill%
			}
		}%
	}\\%
}%
\caption{\label{fig:V}Variances $\text{V}[{F_{\ell}}]$ (\ref{V}) and $\text{V}[{F_{\ell}-F_{\ell-1}}]$ (\ref{dV}) for $\lambda_c=0.1$ (left column), $\lambda_c=0.3$ (middle column), $\lambda_c=0.5$ (right column) and $\nu=0.5$ (top row), $\nu=1$ (middle row), $\nu=2$ (bottom row). The estimated rate $\beta$ corresponds to the rate in assumption~\eqref{A2}.}
\end{figure}

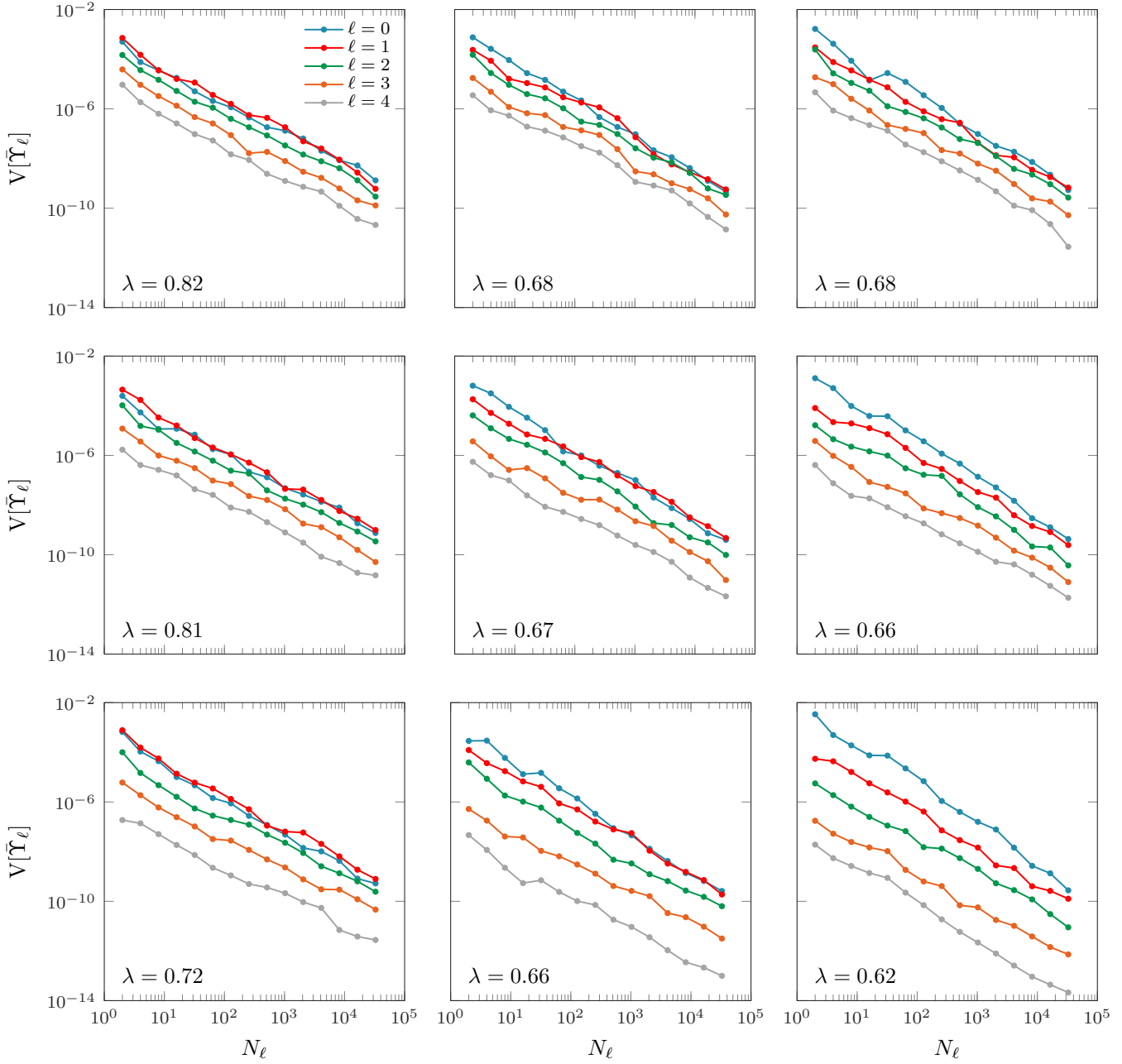
\begin{figure}[p]
\flushleft
\setlength\figurewidth{0.3\textwidth}
\setlength\figureheight{0.3\textwidth}
\foreach \j in {1,...,\n}{
	\foreach \i in {1,...,\n}{
		\pgfmathsetmacro{\case}{int((\j-1)*\n+\i)}
		\ifthenelse{\i>1}{%
			\ifthenelse{\j<3}{%
	\tikzexternalenable
	\tikzsetnextfilename{mse_\case}%
	\begin{tikzpicture}[trim axis left,trim axis right]

\pgfmathsetmacro{\caseMinusOne}{int(\case-1)}

\pgfplotstableread[header=false]{plot_data/table_rates_mse.txt}\rates

\begin{axis}[
width=\figurewidth,
height=\figureheight,
scale only axis,
xmode=log,
xmin=1,
xmax=1e5,
xminorticks=true,
every outer x axis line/.append style={white!15!black},
every x tick label/.append style={font=\color{white!15!black}\scriptsize},
every x label/.append style={font=\color{white!15!black}\scriptsize},
ymode=log,
ymin=1e-14,
ymax=1e-2,
yminorticks=true,
every outer y axis line/.append style={white!15!black},
every y tick label/.append style={font=\color{white!15!black}\scriptsize},
every y label/.append style={font=\color{white!15!black}\scriptsize},
axis background/.style={fill=white},
every y tick label/.append style={opacity=0},
every x tick label/.append style={opacity=0},
legend style={legend cell align=left,align=left,draw=none,font=\scriptsize,at={(0.97,0.97)},anchor=north east,fill=none},
]

\addplot [DottedLineStyle1]
  table[]{plot_data/mse_\case_0.txt};

\addplot [DottedLineStyle2]
  table[]{plot_data/mse_\case_1.txt};

\addplot [DottedLineStyle3]
  table[]{plot_data/mse_\case_2.txt};

\addplot [DottedLineStyle4]
  table[]{plot_data/mse_\case_3.txt};

\addplot [DottedLineStyle5]
  table[]{plot_data/mse_\case_4.txt};

\node[anchor=south west] at (rel axis cs:0.03,0.03) {\small $\lambda=\pgfplotstablegetelem{\caseMinusOne}{0}\of\rates\pgfmathprintnumber[fixed,zerofill]{\pgfplotsretval}$};

\end{axis}
\end{tikzpicture}%
	\tikzexternaldisable
\hfill%
			}{%
	\tikzexternalenable
	\tikzsetnextfilename{mse_\case}%
	\begin{tikzpicture}[trim axis left,trim axis right]

\pgfmathsetmacro{\caseMinusOne}{int(\case-1)}

\pgfplotstableread[header=false]{plot_data/table_rates_mse.txt}\rates

\begin{axis}[
width=\figurewidth,
height=\figureheight,
scale only axis,
xmode=log,
xmin=1,
xmax=1e5,
xminorticks=true,
xlabel={\small $N_\ell$},
every outer x axis line/.append style={white!15!black},
every x tick label/.append style={font=\color{white!15!black}\scriptsize},
every x label/.append style={font=\color{white!15!black}\scriptsize},
ymode=log,
ymin=1e-14,
ymax=1e-2,
yminorticks=true,
every outer y axis line/.append style={white!15!black},
every y tick label/.append style={font=\color{white!15!black}\scriptsize},
every y label/.append style={font=\color{white!15!black}\scriptsize},
axis background/.style={fill=white},
every y tick label/.append style={opacity=0},
legend style={legend cell align=left,align=left,draw=none,font=\scriptsize,at={(0.97,0.97)},anchor=north east,fill=none},
]

\addplot [DottedLineStyle1]
  table[]{plot_data/mse_\case_0.txt};

\addplot [DottedLineStyle2]
  table[]{plot_data/mse_\case_1.txt};

\addplot [DottedLineStyle3]
  table[]{plot_data/mse_\case_2.txt};

\addplot [DottedLineStyle4]
  table[]{plot_data/mse_\case_3.txt};

\addplot [DottedLineStyle5]
  table[]{plot_data/mse_\case_4.txt};

\node[anchor=south west] at (rel axis cs:0.03,0.03) {\small $\lambda=\pgfplotstablegetelem{\caseMinusOne}{0}\of\rates\pgfmathprintnumber[fixed,zerofill]{\pgfplotsretval}$};

\end{axis}
\end{tikzpicture}%
	\tikzexternaldisable
\hfill%
			}
		}{%
			\ifthenelse{\j<3}{%
				\ifthenelse{\j<2}{
	\tikzexternalenable
	\tikzsetnextfilename{mse_\case}%
	\begin{tikzpicture}[trim axis left,trim axis right]

\pgfmathsetmacro{\caseMinusOne}{int(\case-1)}
\pgfplotstableread[header=false]{plot_data/table_rates_mse.txt}\rates

\begin{axis}[
width=\figurewidth,
height=\figureheight,
scale only axis,
xmode=log,
xmin=1,
xmax=1e5,
xminorticks=true,
every outer x axis line/.append style={white!15!black},
every x tick label/.append style={font=\color{white!15!black}\scriptsize},
every x label/.append style={font=\color{white!15!black}\scriptsize},
ymode=log,
ymin=1e-14,
ymax=1e-2,
yminorticks=true,
ylabel={\small $\VarOf{{\MeanQMCEstimatorWithoutSampleReuse}}$},
every outer y axis line/.append style={white!15!black},
every y tick label/.append style={font=\color{white!15!black}\scriptsize},
every y label/.append style={font=\color{white!15!black}\scriptsize},
axis background/.style={fill=white},
every x tick label/.append style={opacity=0},
legend style={legend cell align=left,align=left,draw=none,font=\scriptsize,at={(0.99,0.99)},anchor=north east,fill=none,row sep=-0.1cm},
]

\addplot [DottedLineStyle1]
  table[]{plot_data/mse_\case_0.txt};
\addlegendentry{$\ell=0$};

\addplot [DottedLineStyle2]
  table[]{plot_data/mse_\case_1.txt};
\addlegendentry{$\ell=1$};

\addplot [DottedLineStyle3]
  table[]{plot_data/mse_\case_2.txt};
\addlegendentry{$\ell=2$};

\addplot [DottedLineStyle4]
  table[]{plot_data/mse_\case_3.txt};
\addlegendentry{$\ell=3$};

\addplot [DottedLineStyle5]
  table[]{plot_data/mse_\case_4.txt};
\addlegendentry{$\ell=4$};

\node[anchor=south west] at (rel axis cs:0.03,0.03) {\small $\lambda=\pgfplotstablegetelem{\caseMinusOne}{0}\of\rates\pgfmathprintnumber[fixed,zerofill]{\pgfplotsretval}$};

\end{axis}
\end{tikzpicture}%
	\tikzexternaldisable
\hfill%
				}{%
	\tikzexternalenable
	\tikzsetnextfilename{mse_\case}%
	\begin{tikzpicture}[trim axis left,trim axis right]

\pgfmathsetmacro{\caseMinusOne}{int(\case-1)}

\pgfplotstableread[header=false]{plot_data/table_rates_mse.txt}\rates

\begin{axis}[
width=\figurewidth,
height=\figureheight,
scale only axis,
xmode=log,
xmin=1,
xmax=1e5,
xminorticks=true,
every outer x axis line/.append style={white!15!black},
every x tick label/.append style={font=\color{white!15!black}\scriptsize},
every x label/.append style={font=\color{white!15!black}\scriptsize},
ymode=log,
ymin=1e-14,
ymax=1e-2,
yminorticks=true,
ylabel={\small $\VarOf{{\MeanQMCEstimatorWithoutSampleReuse}}$},
every outer y axis line/.append style={white!15!black},
every y tick label/.append style={font=\color{white!15!black}\scriptsize},
every y label/.append style={font=\color{white!15!black}\scriptsize},
axis background/.style={fill=white},
every x tick label/.append style={opacity=0},
legend style={legend cell align=left,align=left,draw=none,font=\scriptsize,at={(0.97,0.97)},anchor=north east,fill=none},
]

\addplot [DottedLineStyle1]
  table[]{plot_data/mse_\case_0.txt};

\addplot [DottedLineStyle2]
  table[]{plot_data/mse_\case_1.txt};

\addplot [DottedLineStyle3]
  table[]{plot_data/mse_\case_2.txt};

\addplot [DottedLineStyle4]
  table[]{plot_data/mse_\case_3.txt};

\addplot [DottedLineStyle5]
  table[]{plot_data/mse_\case_4.txt};

\node[anchor=south west] at (rel axis cs:0.03,0.03) {\small $\lambda=\pgfplotstablegetelem{\caseMinusOne}{0}\of\rates\pgfmathprintnumber[fixed,zerofill]{\pgfplotsretval}$};

\end{axis}
\end{tikzpicture}%
	\tikzexternaldisable
\hfill%
				}
			}{%
	\tikzexternalenable
	\tikzsetnextfilename{mse_\case}%
	\begin{tikzpicture}[trim axis left,trim axis right]

\pgfmathsetmacro{\caseMinusOne}{int(\case-1)}

\pgfplotstableread[header=false]{plot_data/table_rates_mse.txt}\rates

\begin{axis}[
width=\figurewidth,
height=\figureheight,
scale only axis,
xmode=log,
xmin=1,
xmax=1e5,
xminorticks=true,
xlabel={\small $N_\ell$},
every outer x axis line/.append style={white!15!black},
every x tick label/.append style={font=\color{white!15!black}\scriptsize},
every x label/.append style={font=\color{white!15!black}\scriptsize},
ymode=log,
ymin=1e-14,
ymax=1e-2,
yminorticks=true,
ylabel={\small $\VarOf{{\MeanQMCEstimatorWithoutSampleReuse}}$},
every outer y axis line/.append style={white!15!black},
every y tick label/.append style={font=\color{white!15!black}\scriptsize},
every y label/.append style={font=\color{white!15!black}\scriptsize},
axis background/.style={fill=white},
legend style={legend cell align=left,align=left,draw=none,font=\scriptsize,at={(0.97,0.97)},anchor=north east,fill=none},
]

\addplot [DottedLineStyle1]
  table[]{plot_data/mse_\case_0.txt};

\addplot [DottedLineStyle2]
  table[]{plot_data/mse_\case_1.txt};

\addplot [DottedLineStyle3]
  table[]{plot_data/mse_\case_2.txt};

\addplot [DottedLineStyle4]
  table[]{plot_data/mse_\case_3.txt};

\addplot [DottedLineStyle5]
  table[]{plot_data/mse_\case_4.txt};

\node[anchor=south west] at (rel axis cs:0.03,0.03) {\small $\lambda=\pgfplotstablegetelem{\caseMinusOne}{0}\of\rates\pgfmathprintnumber[fixed,zerofill]{\pgfplotsretval}$};

\end{axis}
\end{tikzpicture}%
	\tikzexternaldisable
\hfill%
			}
		}%
	}\\%
}%
\caption{\label{fig:MSE}Variance of the QMC estimator for the expected value of the difference $F_{\ell}-F_{\ell-1}$, averaged over all random shifts, for different levels and as function of the number of samples $N_\ell$ for $\lambda_c=0.1$ (left column), $\lambda_c=0.3$ (middle column), $\lambda_c=0.5$ (right column) and $\nu=0.5$ (top row), $\nu=1$ (middle row), $\nu=2$ (bottom row). The estimated rate $\lambda$ (averaged over all levels) corresponds to the rate in assumption~\eqref{A4}.}
\end{figure}

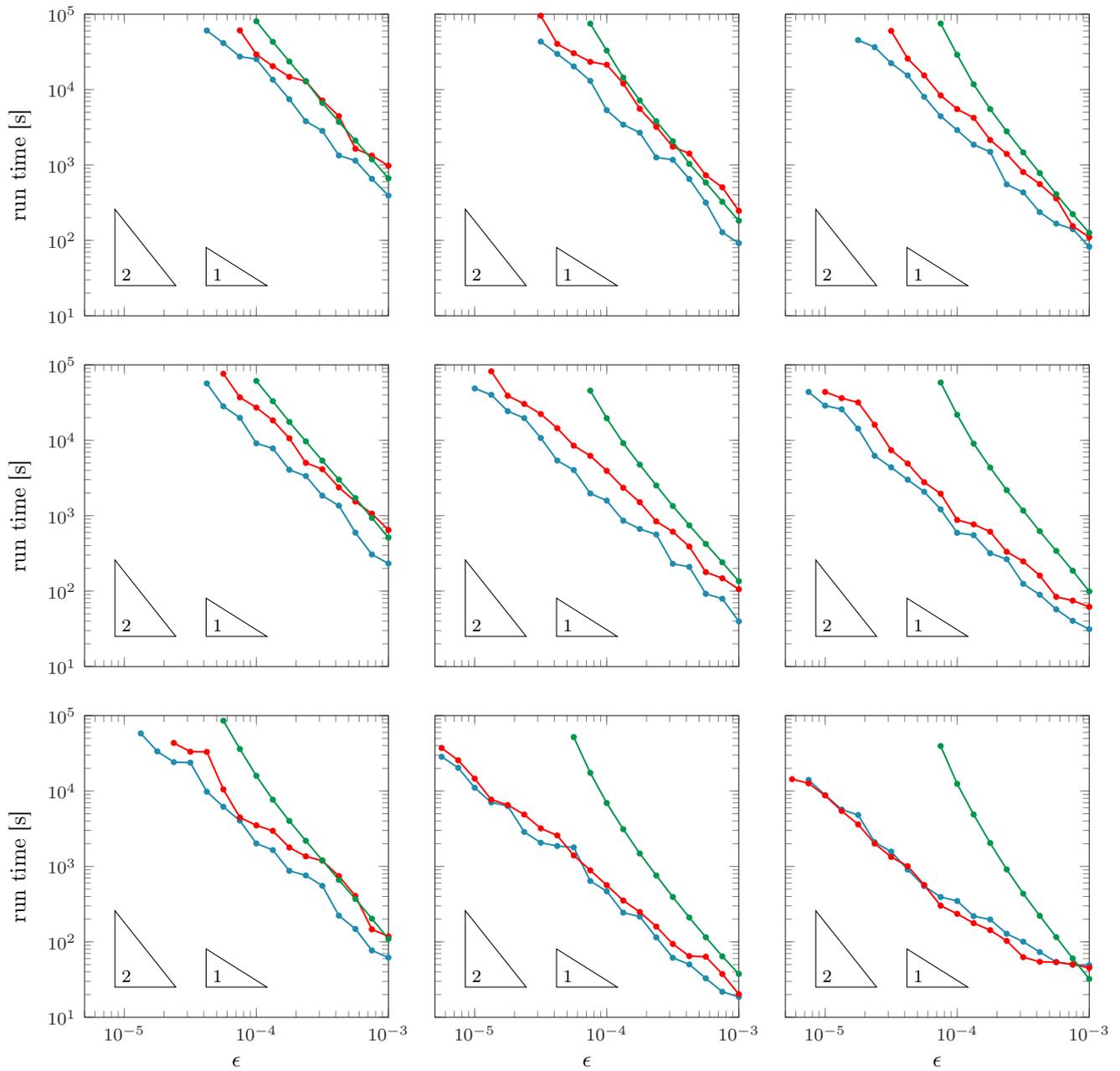
\begin{figure}[p]
\flushleft
\setlength\figurewidth{0.3\textwidth}
\setlength\figureheight{0.3\textwidth}
\foreach \j in {1,...,\n}{
	\foreach \i in {1,...,\n}{
		\pgfmathsetmacro{\case}{int((\j-1)*\n+\i)}
		\ifthenelse{\i>1}{%
			\ifthenelse{\j<3}{%
	\tikzexternalenable
	\tikzsetnextfilename{times_\case}%
	\begin{tikzpicture}[trim axis left,trim axis right]

\pgfmathsetmacro{\caseMinusOne}{int(\case-1)}

\begin{axis}[
width=\figurewidth,
height=\figureheight,
scale only axis,
xmode=log,
xmin=5e-6,
xmax=1e-3,
xminorticks=true,
every outer x axis line/.append style={white!15!black},
every x tick label/.append style={font=\color{white!15!black}\scriptsize},
every x label/.append style={font=\color{white!15!black}\scriptsize},
ymode=log,
ymin=1e1,
ymax=1e5,
yminorticks=true,
every outer y axis line/.append style={white!15!black},
every y tick label/.append style={font=\color{white!15!black}\scriptsize},
every y label/.append style={font=\color{white!15!black}\scriptsize},
every y tick label/.append style={opacity=0},
every x tick label/.append style={opacity=0},
axis background/.style={fill=white},
legend style={legend cell align=left,align=left,draw=none,font=\scriptsize,at={(0.03,0.03)},anchor=south west,fill=none},
]

\addplot [DottedLineStyle1]
  table[]{plot_data/case_\case_1_times.txt};

\addplot [DottedLineStyle2]
  table[]{plot_data/case_\case_2_times.txt};

\addplot [DottedLineStyle3]
  table[]{plot_data/case_\case_3_times.txt};

\LogLogSlopeTriangle{0.3}{0.2}{0.1}{2}{black};
\LogLogSlopeTriangle{0.6}{0.2}{0.1}{1}{black};

\end{axis}
\end{tikzpicture}%
	\tikzexternaldisable
\hfill%
			}{%
	\tikzexternalenable
	\tikzsetnextfilename{times_\case}%
	\begin{tikzpicture}[trim axis left,trim axis right]

\pgfmathsetmacro{\caseMinusOne}{int(\case-1)}

\begin{axis}[
width=\figurewidth,
height=\figureheight,
scale only axis,
xmode=log,
xmin=5e-6,
xmax=1e-3,
xminorticks=true,
xlabel={\small $\epsilon$},
every outer x axis line/.append style={white!15!black},
every x tick label/.append style={font=\color{white!15!black}\scriptsize},
every x label/.append style={font=\color{white!15!black}\scriptsize},
ymode=log,
ymin=1e1,
ymax=1e5,
yminorticks=true,
every outer y axis line/.append style={white!15!black},
every y tick label/.append style={font=\color{white!15!black}\scriptsize},
every y label/.append style={font=\color{white!15!black}\scriptsize},
every y tick label/.append style={opacity=0},
axis background/.style={fill=white},
legend style={legend cell align=left,align=left,draw=none,font=\scriptsize,at={(0.03,0.03)},anchor=south west,fill=none},
]

\addplot [DottedLineStyle1]
  table[]{plot_data/case_\case_1_times.txt};

\addplot [DottedLineStyle2]
  table[]{plot_data/case_\case_2_times.txt};

\addplot [DottedLineStyle3]
  table[]{plot_data/case_\case_3_times.txt};

\LogLogSlopeTriangle{0.3}{0.2}{0.1}{2}{black};
\LogLogSlopeTriangle{0.6}{0.2}{0.1}{1}{black};

\end{axis}
\end{tikzpicture}%
	\tikzexternaldisable
\hfill%
			}
		}{%
			\ifthenelse{\j<3}{%
	\tikzexternalenable
	\tikzsetnextfilename{times_\case}%
	\begin{tikzpicture}[trim axis left,trim axis right]

\pgfmathsetmacro{\caseMinusOne}{int(\case-1)}

\begin{axis}[
width=\figurewidth,
height=\figureheight,
scale only axis,
xmode=log,
xmin=5e-6,
xmax=1e-3,
xminorticks=true,
every outer x axis line/.append style={white!15!black},
every x tick label/.append style={font=\color{white!15!black}\scriptsize},
every x label/.append style={font=\color{white!15!black}\scriptsize},
ymode=log,
ymin=1e1,
ymax=1e5,
yminorticks=true,
ylabel={\small run time [s]},
every outer y axis line/.append style={white!15!black},
every y tick label/.append style={font=\color{white!15!black}\scriptsize},
every y label/.append style={font=\color{white!15!black}\scriptsize},
every x tick label/.append style={opacity=0},
axis background/.style={fill=white},
legend style={legend cell align=left,align=left,draw=none,font=\scriptsize,at={(0.03,0.03)},anchor=south west,fill=none},
]

\addplot [DottedLineStyle1]
  table[]{plot_data/case_\case_1_times.txt};

\addplot [DottedLineStyle2]
  table[]{plot_data/case_\case_2_times.txt};

\addplot [DottedLineStyle3]
  table[]{plot_data/case_\case_3_times.txt};

\LogLogSlopeTriangle{0.3}{0.2}{0.1}{2}{black};
\LogLogSlopeTriangle{0.6}{0.2}{0.1}{1}{black};

\end{axis}
\end{tikzpicture}%
	\tikzexternaldisable
\hfill%
			}{%
	\tikzexternalenable
	\tikzsetnextfilename{times_\case}%
	\begin{tikzpicture}[trim axis left,trim axis right]

\pgfmathsetmacro{\caseMinusOne}{int(\case-1)}

\begin{axis}[
width=\figurewidth,
height=\figureheight,
scale only axis,
xmode=log,
xmin=5e-6,
xmax=1e-3,
xminorticks=true,
xlabel={\small $\epsilon$},
every outer x axis line/.append style={white!15!black},
every x tick label/.append style={font=\color{white!15!black}\scriptsize},
every x label/.append style={font=\color{white!15!black}\scriptsize},
ymode=log,
ymin=1e1,
ymax=1e5,
yminorticks=true,
ylabel={\small run time [s]},
every outer y axis line/.append style={white!15!black},
every y tick label/.append style={font=\color{white!15!black}\scriptsize},
every y label/.append style={font=\color{white!15!black}\scriptsize},
axis background/.style={fill=white},
legend style={legend cell align=left,align=left,draw=none,font=\scriptsize,at={(0.03,0.03)},anchor=south west,fill=none},
]

\addplot [DottedLineStyle1]
  table[]{plot_data/case_\case_1_times.txt};
\label{mgmlqmc};

\addplot [DottedLineStyle2]
  table[]{plot_data/case_\case_2_times.txt};
\label{mlqmc};

\addplot [DottedLineStyle3]
  table[]{plot_data/case_\case_3_times.txt};
\label{mlmc};

\LogLogSlopeTriangle{0.3}{0.2}{0.1}{2}{black};
\LogLogSlopeTriangle{0.6}{0.2}{0.1}{1}{black};

\end{axis}
\end{tikzpicture}%
	\tikzexternaldisable
\hfill%
			}
		}%
	}\\%
}%
\caption{\label{fig:times}Comparison of the run time between MLMC (\ref{mlmc}), MLQMC (\ref{mlqmc}) and MG-MLQMC (\ref{mgmlqmc}) for $\lambda_c=0.1$ (left column), $\lambda_c=0.3$ (middle column), $\lambda_c=0.5$ (right column) and $\nu=0.5$ (top row), $\nu=1$ (middle row), $\nu=2$ (bottom row).}
\end{figure}

\begin{figure}[p]
\flushleft
\setlength\figurewidth{0.3\textwidth}
\setlength\figureheight{0.3\textwidth}
\foreach \j in {1,...,\n}{
	\foreach \i in {1,...,\n}{
		\pgfmathsetmacro{\case}{int((\j-1)*\n+\i)}
		\ifthenelse{\i>1}{%
			\ifthenelse{\j<3}{%
	\tikzexternalenable
	\tikzsetnextfilename{nsamples_\case}%
%
\begin{tikzpicture}[trim axis left, trim axis right]

\begin{axis}[
width=\figurewidth,
height=\figureheight,
scale only axis,
log origin=infty,
xmin=-0.3,
xmax=7.3,
xtick={0,...,7},
every outer x axis line/.append style={white!15!black},
every x tick label/.append style={font=\color{white!15!black}\scriptsize},
every x label/.append style={font=\color{white!15!black}\scriptsize},
ymode=log,
ymin=0.1,
ymax=1e7,
ytick={1e0,1e2,1e4,1e6},
yminorticks=true,
every outer y axis line/.append style={white!15!black},
every y tick label/.append style={font=\color{white!15!black}\scriptsize},
every y label/.append style={font=\color{white!15!black}\scriptsize},
axis background/.style={fill=white},
legend style={legend cell align=left,align=left,font=\tiny,draw=none,at={(1.03,1.03)},anchor=north east},
every y tick label/.append style={opacity=0},
every x tick label/.append style={opacity=0},
x axis line style= { draw opacity=0 },
xtick style={draw=none},
axis y line*=left,
y axis line style= { draw opacity=0 },
ytick style={draw=none},
ymajorgrids,
grid style={line width=1pt,white},
axis on top,
]
\addplot[area legend,draw opacity=0,ybar,bar width=0.6,draw=none,fill=BarColorTwo] plot table[] {plot_data/case_\case_samples_orig.txt};

\addplot[area legend,draw opacity=0,ybar,bar width=0.6,draw=none,fill=BarColorOne] plot table[] {plot_data/case_\case_samples_reused.txt};

\pgfplotstableread[header=false]{plot_data/case_\case_fractions.txt}\fractions
\pgfplotstablegetrowsof{\fractions}
\pgfmathsetmacro{\rows}{\pgfplotsretval-1}

\pgfplotstableread[header=false]{plot_data/case_\case_samples_orig.txt}\total

\pgfplotsset{
    after end axis/.code={
    	\foreach \j in {0,...,\rows} {
			\pgfplotstablegetelem{\j}{0}\of\fractions
   		\pgfmathsetmacro{\a}{\pgfplotsretval}
			\pgfplotstablegetelem{\j}{1}\of\total
   		\pgfmathsetmacro{\b}{\pgfplotsretval}
        	\node[right, align=left, text=black, anchor=south]
			at (axis cs:\j,\b) {\tiny\; \a\%};
		}
    }
}
\end{axis}
\end{tikzpicture}%
%
	\tikzexternaldisable
\hfill%
			}{%
	\tikzexternalenable
	\tikzsetnextfilename{nsamples_\case}%
%
\begin{tikzpicture}[trim axis left, trim axis right]

\begin{axis}[
width=\figurewidth,
height=\figureheight,
scale only axis,
log origin=infty,
xmin=-0.3,
xmax=7.3,
xlabel={\small level},
xtick={0,...,7},
every outer x axis line/.append style={white!15!black},
every x tick label/.append style={font=\color{white!15!black}\scriptsize},
every x label/.append style={font=\color{white!15!black}\scriptsize},
ymode=log,
ymin=0.1,
ymax=1e7,
ytick={1e0,1e2,1e4,1e6},
yminorticks=true,
every outer y axis line/.append style={white!15!black},
every y tick label/.append style={font=\color{white!15!black}\scriptsize},
every y label/.append style={font=\color{white!15!black}\scriptsize},
every y tick label/.append style={opacity=0},
axis background/.style={fill=white},
legend style={legend cell align=left,align=left,font=\tiny,draw=none,at={(1.03,1.03)},anchor=north east},
x axis line style= { draw opacity=0 },
xtick style={draw=none},
axis y line*=left,
y axis line style= { draw opacity=0 },
ytick style={draw=none},
ymajorgrids,
grid style={line width=1pt,white},
axis on top,
]
\addplot[area legend,draw opacity=0,ybar,bar width=0.6,draw=none,fill=BarColorTwo] plot table[] {plot_data/case_\case_samples_orig.txt};

\addplot[area legend,draw opacity=0,ybar,bar width=0.6,draw=none,fill=BarColorOne] plot table[] {plot_data/case_\case_samples_reused.txt};

\pgfplotstableread[header=false]{plot_data/case_\case_fractions.txt}\fractions
\pgfplotstablegetrowsof{\fractions}
\pgfmathsetmacro{\rows}{\pgfplotsretval-1}

\pgfplotstableread[header=false]{plot_data/case_\case_samples_orig.txt}\total

\pgfplotsset{
    after end axis/.code={
    	\foreach \j in {0,...,\rows} {
			\pgfplotstablegetelem{\j}{0}\of\fractions
   		\pgfmathsetmacro{\a}{\pgfplotsretval}
			\pgfplotstablegetelem{\j}{1}\of\total
   		\pgfmathsetmacro{\b}{\pgfplotsretval}
        	\node[right, align=left, text=black, anchor=south]
			at (axis cs:\j,\b) {\tiny\; \a\%};
		}
    }
}
\end{axis}
\end{tikzpicture}%
%
	\tikzexternaldisable
\hfill%
			}
		}{%
			\ifthenelse{\j<3}{%
	\tikzexternalenable
	\tikzsetnextfilename{nsamples_\case}%
%
\begin{tikzpicture}[trim axis left, trim axis right]

\begin{axis}[
width=\figurewidth,
height=\figureheight,
scale only axis,
log origin=infty,
xmin=-0.3,
xmax=7.3,
every x tick label/.append style={opacity=0},
xtick={0,...,7},
every outer x axis line/.append style={white!15!black},
every x tick label/.append style={font=\color{white!15!black}\scriptsize},
every x label/.append style={font=\color{white!15!black}\scriptsize},
ymode=log,
ymin=0.1,
ymax=1e7,
ytick={1e0,1e2,1e4,1e6},
yminorticks=true,
every outer y axis line/.append style={white!15!black},
every y tick label/.append style={font=\color{white!15!black}\scriptsize},
every y label/.append style={font=\color{white!15!black}\scriptsize},
ylabel={\small number of samples},
axis background/.style={fill=white},
legend style={legend cell align=left,align=left,font=\tiny,draw=none,at={(1.03,1.03)},anchor=north east},
x axis line style= { draw opacity=0 },
xtick style={draw=none},
axis y line*=left,
y axis line style= { draw opacity=0 },
ytick style={draw=none},
ymajorgrids,
grid style={line width=1pt,white},
axis on top,
]
\addplot[area legend,draw opacity=0,ybar,bar width=0.6,draw=none,fill=BarColorTwo] plot table[] {plot_data/case_\case_samples_orig.txt};

\addplot[area legend,draw opacity=0,ybar,bar width=0.6,draw=none,fill=BarColorOne] plot table[] {plot_data/case_\case_samples_reused.txt};

\pgfplotstableread[header=false]{plot_data/case_\case_fractions.txt}\fractions
\pgfplotstablegetrowsof{\fractions}
\pgfmathsetmacro{\rows}{\pgfplotsretval-1}

\pgfplotstableread[header=false]{plot_data/case_\case_samples_orig.txt}\total

\pgfplotsset{
    after end axis/.code={
    	\foreach \j in {0,...,\rows} {
			\pgfplotstablegetelem{\j}{0}\of\fractions
   		\pgfmathsetmacro{\a}{\pgfplotsretval}
			\pgfplotstablegetelem{\j}{1}\of\total
   		\pgfmathsetmacro{\b}{\pgfplotsretval}
        	\node[right, align=left, text=black, anchor=south]
			at (axis cs:\j,\b) {\tiny\; \a\%};
		}
    }
}
\end{axis}
\end{tikzpicture}%
%
	\tikzexternaldisable
\hfill%
			}{%
	\tikzexternalenable
	\tikzsetnextfilename{nsamples_\case}%
%
\begin{tikzpicture}[trim axis left, trim axis right]

\begin{axis}[
width=\figurewidth,
height=\figureheight,
scale only axis,
log origin=infty,
xmin=-0.3,
xmax=7.3,
xlabel={\small level},
xtick={0,...,7},
every outer x axis line/.append style={white!15!black},
every x tick label/.append style={font=\color{white!15!black}\scriptsize},
every x label/.append style={font=\color{white!15!black}\scriptsize},
ymode=log,
ymin=0.1,
ymax=1e7,
ytick={1e0,1e2,1e4,1e6},
yminorticks=true,
every outer y axis line/.append style={white!15!black},
every y tick label/.append style={font=\color{white!15!black}\scriptsize},
every y label/.append style={font=\color{white!15!black}\scriptsize},
ylabel={\small number of samples},
axis background/.style={fill=white},
legend style={legend cell align=left,align=left,font=\tiny,draw=none,at={(1.03,1.03)},anchor=north east},
x axis line style= { draw opacity=0 },
xtick style={draw=none},
axis y line*=left,
y axis line style= { draw opacity=0 },
ytick style={draw=none},
ymajorgrids,
grid style={line width=1pt,white},
axis on top,
]
\addplot[area legend,draw opacity=0,ybar,bar width=0.6,draw=none,fill=BarColorTwo] plot table[] {plot_data/case_\case_samples_orig.txt};
\label{original};

\addplot[area legend,draw opacity=0,ybar,bar width=0.6,draw=none,fill=BarColorOne] plot table[] {plot_data/case_\case_samples_reused.txt};
\label{reused};

\pgfplotstableread[header=false]{plot_data/case_\case_fractions.txt}\fractions
\pgfplotstablegetrowsof{\fractions}
\pgfmathsetmacro{\rows}{\pgfplotsretval-1}

\pgfplotstableread[header=false]{plot_data/case_\case_samples_orig.txt}\total

\pgfplotsset{
    after end axis/.code={
    	\foreach \j in {0,...,\rows} {
			\pgfplotstablegetelem{\j}{0}\of\fractions
   		\pgfmathsetmacro{\a}{\pgfplotsretval}
			\pgfplotstablegetelem{\j}{1}\of\total
   		\pgfmathsetmacro{\b}{\pgfplotsretval}
        	\node[right, align=left, text=black, anchor=south]
			at (axis cs:\j,\b) {\tiny\; \a\%};
		}
    }
}
\end{axis}
\end{tikzpicture}%
%
	\tikzexternaldisable
\hfill%
			}
		}%
	}\\%
}%
\caption{\label{fig:samples}Number of original (\ref{original}) and recycled (\ref{reused}) samples taken on each level in the MG-MLQMC algorithm for $\lambda_c=0.1$ (left column), $\lambda_c=0.3$ (middle column), $\lambda_c=0.5$ (right column) and $\nu=0.5$ (top row), $\nu=1$ (middle row), $\nu=2$ (bottom row). The numbers indicate the percentage of recycled samples.}
\end{figure}
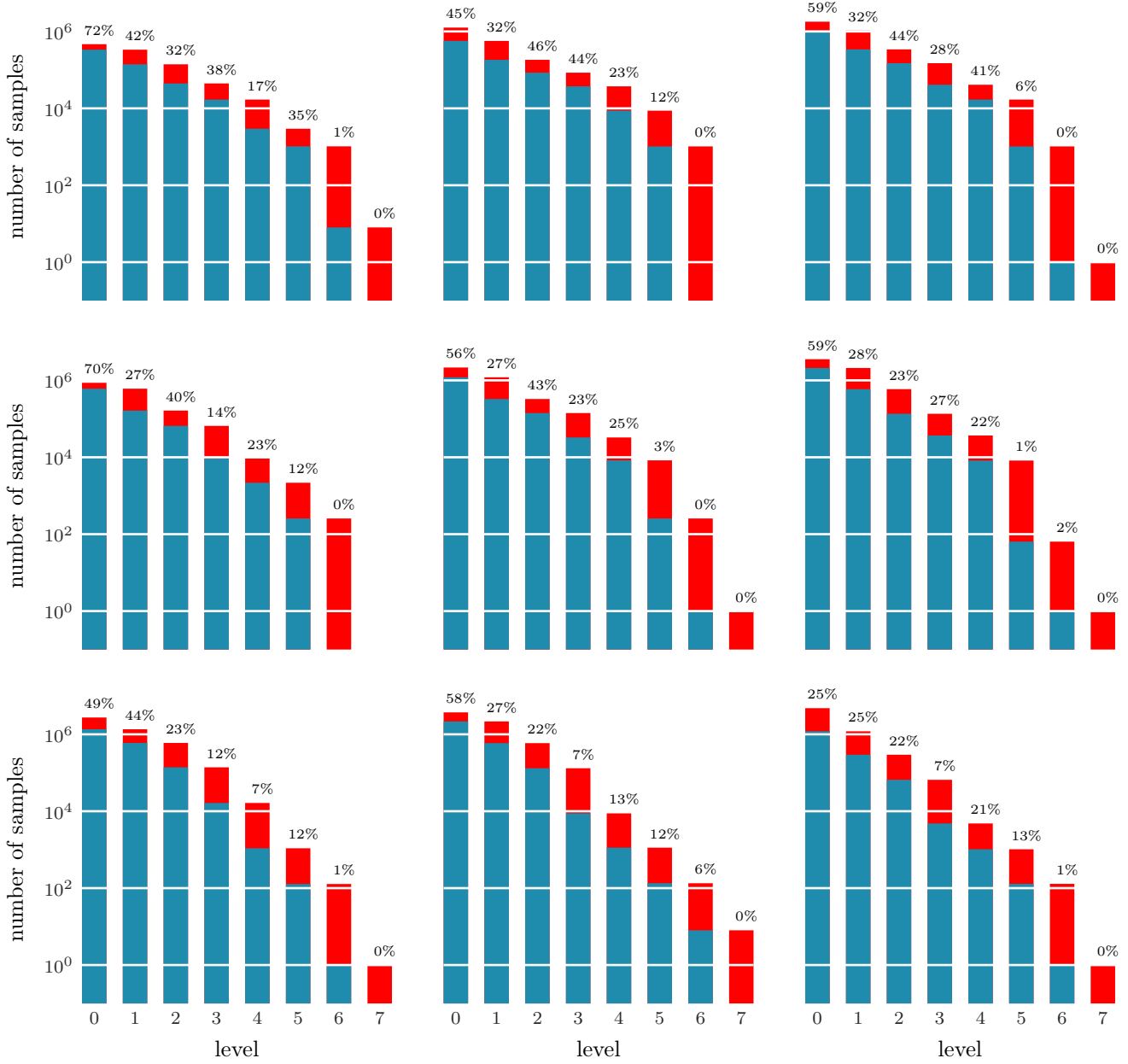

\clearpage

\def\MyColWidth{1.25cm}
\begin{table}[t]
\centering
$\begin{array}{ccR{\MyColWidth}R{\MyColWidth}R{\MyColWidth}}\toprule
\lambda_c & \nu & $S_1$ & $S_2$ & $S_3$ \\ \midrule
\forloop{i}{1}{\value{i} < \n}{
\forloop{j}{1}{\value{j} < \n}{
		\pgfmathsetmacro{\idx}{int((\arabic{i}-1)*\n+\arabic{j}-1)}
		\pgfplotstablegetelem{\idx}{0}\of\params\pgfmathprintnumber[fixed,zerofill,precision=1]{\pgfplotsretval} &
		\pgfmathsetmacro{\idx}{int((\arabic{i}-1)*\n+\arabic{j}-1)}
		\pgfplotstablegetelem{\idx}{1}\of\params\pgfmathprintnumber[fixed,zerofill,precision=1]{\pgfplotsretval} & 
		\pgfmathsetmacro{\idx}{int((\arabic{i}-1)*\n+\arabic{j}-1)}
		\pgfplotstablegetelem{\idx}{0}\of\speedup\pgfmathprintnumber[fixed,zerofill,precision=2]{\pgfplotsretval} & 
		\pgfmathsetmacro{\idx}{int((\arabic{i}-1)*\n+\arabic{j}-1)}
		\pgfplotstablegetelem{\idx}{1}\of\speedup\pgfmathprintnumber[fixed,zerofill,precision=2]{\pgfplotsretval} & 
		\pgfmathsetmacro{\idx}{int((\arabic{i}-1)*\n+\arabic{j}-1)}
		\pgfplotstablegetelem{\idx}{2}\of\speedup\pgfmathprintnumber[fixed,zerofill,precision=2]{\pgfplotsretval}\tabularnewline
}
		\pgfmathsetmacro{\idx}{int((\arabic{i}-1)*\n+\arabic{j}-1)}
		\pgfplotstablegetelem{\idx}{0}\of\params\pgfmathprintnumber[fixed,zerofill,precision=1]{\pgfplotsretval} &
		\pgfmathsetmacro{\idx}{int((\arabic{i}-1)*\n+\arabic{j}-1)}
		\pgfplotstablegetelem{\idx}{1}\of\params\pgfmathprintnumber[fixed,zerofill,precision=1]{\pgfplotsretval} & 
		\pgfmathsetmacro{\idx}{int((\arabic{i}-1)*\n+\arabic{j}-1)}
		\pgfplotstablegetelem{\idx}{0}\of\speedup\pgfmathprintnumber[fixed,zerofill,precision=2]{\pgfplotsretval} & 
		\pgfmathsetmacro{\idx}{int((\arabic{i}-1)*\n+\arabic{j}-1)}
		\pgfplotstablegetelem{\idx}{1}\of\speedup\pgfmathprintnumber[fixed,zerofill,precision=2]{\pgfplotsretval} & 
		\pgfmathsetmacro{\idx}{int((\arabic{i}-1)*\n+\arabic{j}-1)}
		\pgfplotstablegetelem{\idx}{2}\of\speedup\pgfmathprintnumber[fixed,zerofill,precision=2]{\pgfplotsretval}\tabularnewline\midrule
}
\forloop{j}{1}{\value{j} < \n}{
		\pgfmathsetmacro{\idx}{int((\arabic{i}-1)*\n+\arabic{j}-1)}
		\pgfplotstablegetelem{\idx}{0}\of\params\pgfmathprintnumber[fixed,zerofill,precision=1]{\pgfplotsretval} &
		\pgfmathsetmacro{\idx}{int((\arabic{i}-1)*\n+\arabic{j}-1)}
		\pgfplotstablegetelem{\idx}{1}\of\params\pgfmathprintnumber[fixed,zerofill,precision=1]{\pgfplotsretval} & 
		\pgfmathsetmacro{\idx}{int((\arabic{i}-1)*\n+\arabic{j}-1)}
		\pgfplotstablegetelem{\idx}{0}\of\speedup\pgfmathprintnumber[fixed,zerofill,precision=2]{\pgfplotsretval} & 
		\pgfmathsetmacro{\idx}{int((\arabic{i}-1)*\n+\arabic{j}-1)}
		\pgfplotstablegetelem{\idx}{1}\of\speedup\pgfmathprintnumber[fixed,zerofill,precision=2]{\pgfplotsretval} & 
		\pgfmathsetmacro{\idx}{int((\arabic{i}-1)*\n+\arabic{j}-1)}
		\pgfplotstablegetelem{\idx}{2}\of\speedup\pgfmathprintnumber[fixed,zerofill,precision=2]{\pgfplotsretval}\tabularnewline
}
		\pgfmathsetmacro{\idx}{int((\arabic{i}-1)*\n+\arabic{j}-1)}
		\pgfplotstablegetelem{\idx}{0}\of\params\pgfmathprintnumber[fixed,zerofill,precision=1]{\pgfplotsretval} &
		\pgfmathsetmacro{\idx}{int((\arabic{i}-1)*\n+\arabic{j}-1)}
		\pgfplotstablegetelem{\idx}{1}\of\params\pgfmathprintnumber[fixed,zerofill,precision=1]{\pgfplotsretval} & 
		\pgfmathsetmacro{\idx}{int((\arabic{i}-1)*\n+\arabic{j}-1)}
		\pgfplotstablegetelem{\idx}{0}\of\speedup\pgfmathprintnumber[fixed,zerofill,precision=2]{\pgfplotsretval} & 
		\pgfmathsetmacro{\idx}{int((\arabic{i}-1)*\n+\arabic{j}-1)}
		\pgfplotstablegetelem{\idx}{1}\of\speedup\pgfmathprintnumber[fixed,zerofill,precision=2]{\pgfplotsretval} & 
		\pgfmathsetmacro{\idx}{int((\arabic{i}-1)*\n+\arabic{j}-1)}
		\pgfplotstablegetelem{\idx}{2}\of\speedup\pgfmathprintnumber[fixed,zerofill,precision=2]{\pgfplotsretval}\tabularnewline\bottomrule
\end{array}$
\caption{\label{tab:speedup}Speedup of MLQMC over MLMC ($S_1$) and MG-MLQMC over MLMC ($S_2$) and MLQMC ($S_3$).}
\end{table}

\subsection{Results and comparison}

We compare the performance of MLQMC, with and without recycling, against a basic MLMC simulation. We ran all simulations for a sequence of decreasing tolerances, starting from $\epsilon=0.001$. The results are shown in~\figref{fig:times}. Observe that the asymptotic complexity rate of the QMC methods is increasingly better (that is, the slope of the lines decrease) as the smoothness and correlation length in the Mat\'ern covariance function increase, i.e., when moving from top to bottom and left to right. In particular, this means that fewer computational resources are needed to achieve a certain tolerance on the MSE~\eqref{eq:MSE} (note the logarithmic time axis). Also, and consistent with our previous analysis, we observe that the benefit of the sample recycling decreases accordingly. That is, the lines for MG-MLQMC and MLQMC are closer together when moving towards the bottom right of~\figref{fig:times}.

These results are summarized in~\tabref{tab:speedup}, where the speedup is shown for all parameter combinations, averaged over all tolerances $\epsilon$. The first column shows the speedup of MLQMC over MLMC, i.e,
\begin{equation}
S_1 = \frac{\mathrm{time}({{\mathcal{Q}}^\text{MC}_{L}})}{\mathrm{time}({{\bar{\mathcal{Q}}}^\text{QMC}_{L,\nbshifts}})}. \nonumber
\end{equation}
Columns two and three show
\begin{equation}
S_2 = \frac{\mathrm{time}({{\mathcal{Q}}^\text{MC}_{L}})}{\mathrm{time}({{\bar{\mathcal{Q}}}^\text{QMC}_{L,\nbshifts,\text{reuse}}})} \quad \textrm{and} \quad S_3 = \frac{\mathrm{time}({{\bar{\mathcal{Q}}}^\text{QMC}_{L,\nbshifts}})}{\mathrm{time}({{\bar{\mathcal{Q}}}^\text{QMC}_{L,\nbshifts,\text{reuse}}})}, \nonumber
\end{equation}
the speedup of MG-MLQMC over MLMC, and MG-MLQMC over MLQMC respectively. From this table, we conclude that the MLQMC algorithm is a very efficient algorithm compared to the basic MLMC method. The method performs better as the covariance function becomes smoother. However, by using an FMG solver and reusing samples on coarser levels, we can gain another factor two in computational work reduction.

To better understand the benefit of the sample recycling, we show the sample distribution across the levels for the finest computed tolerance in the MG-MLQMC method in~\figref{fig:samples}. First, notice that the $N_\ell$ is a decreasing sequence with $\ell$, as required. Secondly, the number of samples decreases faster when moving from top to bottom, indicating that there are fewer samples to be reused on coarser levels. Also indicated on the figure is the percentage of recycled samples, i.e., $N_\ell/N_{\ell+1}\cdot100\%$. The higher this value, the better the method with sample recycling will perform, compared to standard MLQMC without recycling. These numbers should be compared to the results in~\figref{fig:times} and~\tabref{tab:speedup}.

\section{Conclusions and future work}

In this work, we have introduced an algorithm that combines MLMC with a Multigrid method, similar to the work in~\cite{kumar2017multigrid}. Our goal was to recycle coarse approximations from the FMG method as samples in the MLMC method, in such a way that no additional error is introduced.  We presented a practical algorithm that achieves this goal in~\secref{sec:practical} and analyzed its cost in~\secref{sec:cost_anal}. 

In the numerical experiments with our new method, we confined ourselves to the well-known two-dimensional elliptic PDE with random coefficients, where the uncertain coefficients are modeled as a lognormal random field with a covariance function of Mat\'ern type. We performed various simulations with different combinations of correlation length and smoothness in this covariance function. We observed that, as the problem becomes more difficult (i.e., the covariance function becomes less smooth), there is more gain from the sample recycling. Hence, our method works in tandem with the QMC method, which works better as the smoothness parameter in the covariance function increases.

A next step would be to replace the Multigrid method by an Algebraic Multigrid (AMG) method. Contrary to the geometric Multigrid method we used in this work, AMG does not require a physical grid to be associated with each level. Instead, coarsening happens fully based on the matrix entries, resulting in a completely autonomous solver.

Another way to combine MLMC with Multigrid is to use the Multi-Index setting from~\cite{haji2016multi}, where multiple directions of refinement are allowed. One direction would then correspond to the physical discretization, the other direction would be the number of V-cycles in the Multigrid method. This could then be combined with the adaptive approach from~\cite{robbe2017dimension}.

\section*{Acknowledgments}
This research was funded by project IWT/SBO EUFORIA: ``Efficient Uncertainty quantification For Optimization in Robust design of Industrial Applications'' (IWT-140068) of the Agency for Innovation by Science and Technology, Flanders, Belgium. We gratefully acknowledge the financial support from the Statistical and Applied Mathematical Sciences Institute (SAMSI) 2017 Year-long Program on Quasi-Monte Carlo and High-Dimensional Sampling Methods for Applied Mathematics.

\bibliographystyle{plain}
\bibliography{references}

\end{document}